%% file: archiv_main.tex
\documentclass[a4,10pt,twocolumn]{scrartcl}
\usepackage[top=2.5cm, bottom=2.5cm, left=1.6cm, right=1.6cm]{geometry}

\pdfoutput=1

\usepackage[english]{babel}
\usepackage[utf8x]{inputenc}
\usepackage{enumitem}					

\usepackage{libertine}
\usepackage[T1]{fontenc}
\usepackage[noBBpl]{mathpazo}

\usepackage[hang]{footmisc}
\setfnsymbol{wiley}

\makeatletter
\renewcommand{\@biblabel}[1]{[#1]\hfill}
\makeatother

\makeatletter
\let\NAT@parse\undefined
\makeatother

\usepackage[format=plain,			
            labelsep=period,		
            font=small,			
            labelfont=bf,			
            skip=5pt			
            ]{caption}
            
\usepackage[final,kerning=false,protrusion=false]{microtype}

\usepackage{mathtools}
\usepackage{amsmath}
\usepackage{amsthm}
\usepackage{amssymb}
\usepackage[normalem]{ulem}
\usepackage{enumitem} 
\usepackage[makeroom]{cancel}
\usepackage{graphics} 
\usepackage{epsfig} 
\usepackage[dvipsnames,table,xcdraw]{xcolor}
\usepackage{tikz}
\usepackage{cite}

\usepackage{graphicx}
\usepackage{algorithm}
\usepackage{algcompatible}
\usepackage{algpseudocode}
\usepackage{pgfplots}
\pgfplotsset{compat=1.7}
 
\usepackage{authblk}

\newcommand*{\rom}[1]{\expandafter\@slowromancap\romannumeral #1@}
\makeatother
\newtheorem{ass}{Assumption}
\newtheorem{lem}{Lemma}

\newtheorem{theorem}{Theorem}
\newtheorem{rem}{Remark}
\newtheorem{corollary}{Corollary}
\DeclareMathOperator*{\argmin}{arg\,min}

\let\oldTheorem\theorem
\renewcommand{\theorem}{\oldTheorem\normalfont}
\let\oldLemma\lem
\renewcommand{\lem}{\oldLemma\normalfont}
\let\oldCorollary\corollary
\renewcommand{\corollary}{\oldCorollary\normalfont}
\let\oldRemark\rem
\renewcommand{\rem}{\oldRemark\normalfont}
\let\oldAssumption\ass
\renewcommand{\ass}{\oldAssumption\normalfont}

\addtokomafont{paragraph}{\itshape}

\newcommand{\margin}[1]{\marginpar{\color{red}\tiny\ttfamily#1}}

\renewcommand{\margin}[1]{\marginpar{\color{red}\tiny\ttfamily}}

\title{\LARGE \textbf
Anytime Proximity Moving Horizon Estimation: \\Stability and Regret 
\footnote{This article is a slightly modified version of \cite{gharbi2020anytime}.}
}

\setkomafont{section}{\Large}
\setkomafont{subsection}{\large}
\setkomafont{subsubsection}{\itshape}
\newcommand\rev[1]{{\color{black}{#1}}}

\begin{document}

\date{}
\author[1]{Meriem Gharbi}
\author[2]{Bahman Gharesifard}
\author[1]{Christian Ebenbauer}
\affil[1]{Chair of Intelligent Control Systems, RWTH Aachen University, Germany \protect\\ \texttt{\small $\lbrace$meriem.gharbi,christian.ebenbauer$\rbrace$@ic.rwth-aachen.de}}
\affil[2]{Electrical \& Computer Engineering Department, University of California, Los Angeles, USA \protect\\ \texttt{\small gharesifard@ucla.edu} \protect\\[2em]}

\maketitle

\begin{abstract}
\textbf{Abstract.} In this paper, we address the efficient implementation of moving horizon state estimation of constrained discrete-time  linear systems. We propose a novel iteration scheme which employs a proximity-based formulation of the underlying optimization algorithm and reduces computational effort by performing only a limited number of  optimization iterations each time a new measurement is received.  We outline conditions under which global exponential stability of the underlying estimation errors is ensured.  Performance guarantees  of the iteration scheme in terms of regret upper bounds are also established.  A combined result shows that both exponential stability and a sublinear regret which can be rendered smaller by increasing the number of optimization iterations can be guaranteed. The stability and regret results of the proposed estimator are showcased through numerical simulations.
\end{abstract}

\input{intro}
\input{setup}
\input{algorithm}

\input{stability}

\input{regret}
\input{example}
\input{conc}

\bibliography{bib}
\bibliographystyle{unsrt}

\renewcommand{\thesubsection}{\Alph{subsection}}
\section*{Appendix}
 \input{App}

\end{document}

%% file: intro.tex
\section{Introduction}
Moving horizon estimation (MHE) is an optimization-based state estimation approach that computes an estimate of the state of a dynamical system by using a finite number of the most recent measurements. More specifically, a suitable optimization problem is solved to compute the optimal estimate at each time instant and the horizon of measurements is shifted forward in time whenever a new measurement becomes available. Various MHE formulations have been proposed and investigated for stability and are by now well-established in state estimation area~\cite{rao2001constrained,bookDiehl,alessandri2003receding,sui2014linear,rao2003constrained,haseltine2005critical,rawlings2012optimization}. Practical issues related to the online solution of MHE has drawn special attention since the underlying optimization problem has to be solved online at each time instant. In order to overcome this computational burden, fast optimization strategies based on interior-point methods~\cite{jorgensen2004numerical,haverbeke2009structure} are proposed, however, with no theoretical guarantees. In~\cite{alessandri2008moving},  approximation schemes are considered, in which suboptimal solutions for minimizing quadratic cost functions with a given accuracy  are allowed and upper bounds on the estimation errors are derived under observability assumptions. However, no  optimization algorithm is specified.
 A similar convergence analysis is carried in \cite{zavala2010stability} for the MHE algorithm presented in~\cite{zavala2008fast}, where a nominal background problem is solved based on predicted future measurements and when the true measurement arrives, the actual state is computed using a fast online correction step. To show that the generated  estimation errors remain bounded, the associated approximate cost  and resulting suboptimality are taken into account in the analysis.  Particularly interesting are works which explicitly consider the dynamics of the optimization algorithm in the convergence analysis \cite{alessandri2017fast,alessandri2020fast}. 
 In~\cite{alessandri2017fast}, a fast MHE implementation is achieved by performing single or multiple iterations of gradient or Newton methods to minimize least-squares cost functions.  For linear systems, global exponential stability of the estimation errors is shown based on an explicit representation of the error dynamics. However, the required  observability assumption restricts the choice of the horizon length  and implies that it has to be greater than the state dimension.  Moreover, variants of the so-called real-time iteration scheme~\cite{kuhl2011real} which performs a single  Gauss-Newton iteration per time instant are proposed. 
 The local convergence results derived in~\cite{wynn2014convergence}  are established for the unconstrained case, i.e. no inequality constraints are considered, and hold under the assumptions of observability and a sufficiently small initial estimation  error.  
 Real-time implementations of MHE are also successfully carried out in real-world applications, such as structural vibration applications~\cite{abdollahpouri2017real}, induction machines~\cite{favato2019fast} and industrial separation processes~\cite{kupper2009efficient}.  However, theoretical studies that consider both stability as well as performance 
of MHE schemes under rather mild assumptions are to the best of our knowledge rarely addressed in the literature.

\textbf{Statement of contributions.} In this work, we present a novel MHE iteration scheme for constrained linear discrete-time systems, which is based on
the idea of proximity MHE (pMHE), recently introduced in~\cite{gharbi2018proximity,gharbi2020anIteration}. 
The pMHE framework exploits the advantages of a stabilizing a priori estimate from which stability can provably be inherited for any horizon length while allowing for a flexible design via rather general convex stage costs. \rev{In \cite{gharbi2018proximity}, stability is investigated under the assumption that a solution of the optimization problem is available at each time instant. In \cite{gharbi2020anIteration}, an unconstrained MHE problem in which the inequality constraints are incorporated into the cost function by means of so-called relaxed barrier functions is considered. Although only a limited number of iterations are executed at each time instant, the derived stability conditions tailored to this relaxed barrier function based formulation are rather conservative.}  The contribution  of this paper is fourfold. First, we present a pMHE iteration scheme where at each time instant, a limited number of optimization iterations are carried out and a state estimate is delivered in real-time. The underlying optimization algorithm consists of a proximal point algorithm~\cite{beck2003mirror} and is warm-started by a stabilizing a priori estimate constructed based on the Luenberger observer. Second, we establish global exponential stability  of the underlying estimation errors under minimal assumptions and by means of a Lyapunov analysis. In particular, the iteration scheme can be considered as an \textit{anytime} algorithm in which stability  is guaranteed   after any number of optimization algorithm iterations, including the case of a single iteration per time instant. \rev{Third, in contrast to the pMHE scheme in \cite{gharbi2018proximity}, the a priori estimate is only used to warm-start the proposed algorithm. Nevertheless, stability is inherited from it, despite its stabilizing effect is fading away with each iteration.} Forth, we study the performance of the pMHE iteration scheme by using the notion of regret, which is widely used in the field of online convex optimization to characterize performance~\cite{hazan2008adaptive,hazan2016introduction,hall2013dynamical}, and adapting it to our setting. More specifically, we define the regret as the difference of the accumulated costs generated by the iteration scheme relative to a comparator sequence and show that this regret can be upper bounded. Furthermore, we prove that, for any given comparator sequence, this bound can be rendered smaller by increasing the number of optimization iterations,  and that a constant regret bound can be derived for the special case of exponentially stable comparator sequences.  Overall, we present a novel anytime pMHE iteration scheme that is designed based on rather general convex stage cost functions, ensures stability after each iteration as well as for any horizon length,  and for which performance guarantees are provided and characterized in terms of rigorously derived  regret bounds.

\textbf{Organization.} The paper is  organized as follows. The constrained MHE problem for discrete-time linear systems is stated in Section \ref{sec: setup}. The proposed pMHE iteration scheme is described in details in Section \ref{sec: iteration scheme} and its stability properties are established in Section \ref{sec: stability}. In Section \ref{sec: regret}, the focus is on the performance properties of the iteration scheme which are reflected by the derived regret upper bounds. A simulation example that illustrates both the stability and performance properties is presented in Section \ref{sec: sim}. Finally, conclusions are drawn in Section \ref{sec: conclusion}.
\\
~\\
  \textit{Notation:} Let $\mathbb{N}_{+}$ denote the set of positive natural numbers, $\mathbb{R}_{+}$ and $\mathbb{R}_{+\!\!+}$  the sets of nonnegative real and positive real numbers, respectively, and $\mathbb{S}_{+}^n$ and $\mathbb{S}_{+\!\!+}^n$ the sets of symmetric positive semi-definite and positive-definite matrices of dimension $n \in \mathbb{N}_{+}$, respectively. For a vector $v\in \mathbb{R}^n$, let  $\left\Vert v \right\Vert_P \coloneqq \sqrt{v^{\top}P\,v}$  for any $P \in \mathbb{S}_{+}^n $.  Moreover, let $\mathbf{0} \coloneqq \begin{bmatrix} 0 & \cdots & 0 \end{bmatrix}^\top$. 

%% file: setup.tex
 \section{Problem setup and preliminaries}
 \label{sec: setup}
We consider the following discrete-time linear time-invariant (LTI) system 
 \begin{subequations}
 \label{system}
\begin{align}
x_{k+1}&=A\, x_k + B\, u_k  , \\ 
y_k &= C\,  x_k,
	\end{align}
\end{subequations} 
where $x_k  \in \mathbb{R}^n$ denotes the state vector, $u_k \in \mathbb{R}^m $ the input vector, and $y_k  \in \mathbb{R}^p$ the measurement vector.  We assume that the pair $(A,C)$ is detectable  and that the state satisfies polytopic constraints  
\begin{align}
x_k \in \mathcal{X}\coloneqq\left\lbrace x \in \mathbb{R}^n: C_{\text{x}} \, x \leq d_{\text{x}} \right\rbrace 
\end{align}
where $ C_{\text{x}} \in \mathbb{R}^{q_{\text{x}} \times n}$ and $d_{\text{x}} \in \mathbb{R}^{q_{\text{x}}}$ with $q_{\text{x}} \in \mathbb{N}_+$. We aim to compute an estimate of the state $x_k$ based on a moving horizon estimation scheme. More specifically, at each time instant $k$,  given the last $N$ measurements $\lbrace y_{k-N}, \cdots, y_{k-1} \rbrace$ and inputs $\lbrace u_{k-N}, \cdots, u_{k-1} \rbrace$, our goal is to find a solution  to the following optimization problem 
 \begin{subequations}
  \label{miniProx}
\begin{align}
&\min\limits_{\substack{\hat{x}_{k-N} ,\\ \hat{\mathbf{v}}, \hat{\mathbf{w}}  }} \quad \enskip \sum_{i=k-N}^{k-1} r\left(  \hat{v}_i \right) + q\left(\hat{w}_i \right) \label{N>1Unc} \\
&\enskip\text{s.t. } \hspace{0.75cm} \hat{x}_{i} = A\, \hat{x}_{i-1}+ B\, u_{i-1} +\hat{w}_{i-1},  \label{dnamics}\\
&\hspace{1.45cm} y_i=   C\, \hat{x}_i + \hat{v}_i, \label{dnamicsv} \\
&\hspace{1.45cm} \hat{x}_{i} \in \mathcal{X}, \hspace{1.5cm}  i= k-N, \cdots, k-1, \label{constraints}
\end{align}
\end{subequations} 
where $\hat{\mathbf{v}} =\lbrace     \hat{v}_{k-N}, \cdots, \hat{v}_{k-1}\rbrace $ and $\hat{\mathbf{w}} =\lbrace     \hat{w}_{k-N}, \cdots, \hat{w}_{k-1}\rbrace $ denote the output residual and the  model residual sequences over the estimation horizon with length $N \in \mathbb{N}_+$.  In~\eqref{N>1Unc}, the stage cost $r : \mathbb{R}^p \rightarrow \mathbb{R}$ is a 
convex function  which penalizes the output residual $  \hat{v}_i \in \mathbb{R}^p $, and the stage cost $q : \mathbb{R}^n \rightarrow \mathbb{R}$ is a convex function which penalizes the model residual $\hat{w}_i \in \mathbb{R}^n$. By using the system dynamics~\eqref{dnamics} and \eqref{dnamicsv}, we can  express each  output residual $  \hat{v}_i $ in terms of the remaining decision variables $\lbrace \hat{x}_{k-N}, \hat{\mathbf{w}}  \rbrace$, which we collect in the vector 
\begin{equation}\label{eq:z-def}
\hat{\mathbf{z}}_k \coloneqq \begin{bmatrix}
	 \hat{x}_{k-N} \\  \hat{w}_{k-N}\\[-0.3em]\vdots  \\[-0.3em]\hat{w}_{k-1}    \end{bmatrix} \in \mathbb{R}^{(N+1)n}
\end{equation}
and use it to reformulate problem  (\ref{miniProx})  as
\begin{subequations}
\label{condensed}
\begin{align}
\label{N>1UncCom}
&\min_{\hat{\mathbf{z}}_k}  \qquad f_k\left( \hat{\mathbf{z}}_k\right)   
\\& \enskip \text{s.t. }  \qquad \hat{\mathbf{z}}_k  \in \mathcal{S}_k  \label{constraintsCompact}.
\end{align}
\end{subequations}
Here,  the convex function $f_k: \mathbb{R}^{(N+1)n} \rightarrow \mathbb{R}$ denotes the sum of stage costs  and the convex set $ \mathcal{S}_k \subset  \mathbb{R}^{(N+1)n} $ represents the (stacked) state constraints given by 
\begin{align}
\label{stackedConstr}
 \mathcal{S}_k = \left\lbrace \mathbf{z} = \begin{bmatrix}x \\ \mathbf{w} \end{bmatrix}, x \in \mathbb{R}^n, \mathbf{w}\in \mathbb{R}^{Nn}:  G\, x + F\,   \mathbf{w} \leq E_k\right\rbrace.   
\end{align}
 Note that $ \mathcal{S}_k$ is time-dependent due to the changing input sequence $\lbrace u_{k-N}, \cdots, u_{k-1} \rbrace$ that enters $E_k$ over time. The matrices $G$, $F$ and the vector $E_k$ as well as more details on the reformulation of the estimation problem (\ref{miniProx}) to (\ref{condensed}) can be found in Appendix  \ref{app:condensed}.
Within the proximity-based formulation, as introduced in~\cite{gharbi2018proximity,gharbi2019proximity}
and related to \cite{sui2014linear}, we solve a regularized form of   (\ref{condensed})  in which we add to the cost function~\eqref{N>1UncCom}  a proximity measure to a stabilizing a priori estimate, which we refer to as $\bar{\mathbf{z}}_{k} \in  \mathbb{R}^{(N+1)n}$.  A corresponding pseudo-code is given in Algorithm \ref{alg: optimal pMHE}.
\begin{algorithm} 
\caption{pMHE according to \cite{gharbi2018proximity}}
\label{alg: optimal pMHE}
\begin{algorithmic}[1]
\State  \textbf{Initialize:}   Choose  $ \hat{x}_0  $ and set $\bar{\mathbf{z}}_0  =\hat{x}_0$ 
  \For {$k=1,2, \cdots$  } \vspace{0.2cm} 
	\State   \vspace{-0.67cm}   \begin{align} 
  \label{condensedpMHE}
 \hspace{-1.7cm} \mathbf{\hat{z}}^*_{k}=\argmin_{\quad\, \hat{\mathbf{z}}_k\in  \mathcal{S}_k}\left\lbrace f_k\left( \hat{\mathbf{z}}_k \right)   + D_\psi \left( \hat{\mathbf{z}}_k, \bar{\mathbf{z}}_{k} \right)  \right\rbrace   
   \end{align}
		\State obtain $\hat{x}_k$ according to (\ref{StateEstOpt}) 
		\State  $\mathbf{\bar{z}}_{k+1} =\Phi_k \left( \mathbf{\hat{z}}^*_{k} \right)$  
		\EndFor
\end{algorithmic}
\end{algorithm}
~\\In (\ref{condensedpMHE}), the overall cost function is strictly convex and     
\[
D_\psi : \mathbb{R}^{(N+1)n} \times \mathbb{R}^{(N+1)n} \rightarrow \mathbb{R} 
\]
denotes the Bregman distance induced from a continuously differentiable and strongly convex function $\psi : \mathbb{R}^{(N+1)n}   \rightarrow \mathbb{R} $ as 
 \begin{align}
 \label{BregmanDef}
D_\psi(\mathbf{z}_1,\mathbf{z}_2)= \psi(\mathbf{z}_1) - \psi (\mathbf{z}_2)  -  (\mathbf{z}_1-\mathbf{z}_2)^\top \nabla \psi(\mathbf{z}_2).
 \end{align}
 More detail on Bregman distances as well as some of their central properties can be found in Appendix \ref{app: Bregman}. Based on the resulting pMHE solution $ \mathbf{\hat{z}}^*_{k}$, the state estimate $\hat{x}_k$ is obtained via a forward prediction of the dynamics (\ref{dnamics}):
 \begin{align}
		\label{StateEstOpt}
\hat{x}_k =A^N\, \hat{x}^{*}_{k-N}+ \!\!\sum_{j=k-N} ^ {k-1}\! \! A^{k-1-j}\, \left( B\,   u_j + \hat{w}^{*}_{j} \right),
\end{align}
and the stabilizing a priori estimate $\bar{\mathbf{z}}_{k+1}$ is computed using the operator $\Phi_k : \mathbb{R}^{(N+1)n} \rightarrow \mathbb{R}^{(N+1)n}$.   While performance of  pMHE can be enforced with rather general convex stage costs $r$ and $q$, stability  can be ensured for any horizon length $N \geq1$  with an appropriate choice of the a priori estimate operator $\Phi_k$ and the
Bregman distance \cite{gharbi2018proximity}. Furthermore, among many interesting properties, Bregman distances can adapt to the problem at hand and act as a barrier for the constraint set, in particular with so-called relaxed barrier functions \cite{gharbi2020anIteration}.

In the following section, we present an iteration scheme to pMHE, in which,  rather than finding the pMHE solution at each time instant $k$, we reduce the computation time by executing only a finite number of optimization iterations of a gradient type algorithm.

%% file: algorithm.tex
\section{Anytime pMHE Algorithm}
\label{sec: iteration scheme}
In this section, we propose a novel pMHE iteration scheme in which, at each time instant $k$, problem (\ref{condensed}) is approximately solved by executing  a fixed number $\mathrm{it}(k) \in \mathbb{N}_+$ of optimization algorithm iterations. In more details, at each time $k$, a suitable warm start $\mathbf{\hat{z}}^{0}_{k}$ is generated from a stabilizing a priori estimate $\mathbf{\bar{z}}_k$ and an iterative optimization update is  carried out, from which the sequence   $\big\lbrace  \mathbf{\hat{z}}^{i}_{k}   \big\rbrace $ with $i=1,\cdots,\mathrm{it}(k)$ is obtained. The steps of the scheme are given in Algorithm~\ref{alg: eta changes} and illustrated in Figure \ref{Fig: iteration scheme}.
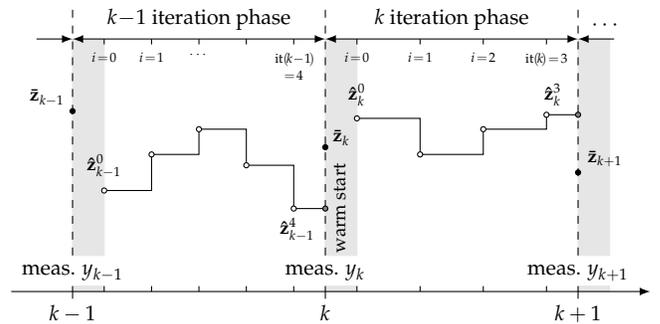
\begin{figure}[ht]
  \centering
\input{figures/iteration_scheme_sketch}
\caption{Illustration of the steps of the pMHE iteration scheme at the time instants $k-1$ and $k$, with the corresponding $\mathrm{it}(k-~1)=4$  and $\mathrm{it}(k)=3$ optimization iterations. }
\label{Fig: iteration scheme}
\end{figure}
~\\Before we explain the proposed  algorithm in more detail, and for the sake of clarity, let us first introduce  some notations. The index $k$ denotes the time instant in which we receive a new measurement and $\mathrm{it}(k)$ is the number of iterations of the optimization algorithm between time instants $k$ and $k+1$. Moreover, we introduce
\begin{align}
\label{notation}
\hspace{-0.4cm}\mathbf{\hat{z}}^{ i}_{k} \coloneqq \begin{bmatrix}
				 \hat{x}^{ i}_{k-N} \\[0.2em]  \hat{w}_{k-N}^{ i} \\\vdots  \\[0.2em]  \hat{w}_{k-1}^{ i}    \end{bmatrix}, \enskip
 \mathbf{\bar{z}}_k \coloneqq \begin{bmatrix}
 \bar{x}_{k-N}\\[0.2em]  \bar{w}_{k-N}\\\vdots  \\[0.2em]  \bar{w}_{k-1}
 \end{bmatrix}, \enskip
     \mathbf{z}_k \coloneqq \begin{bmatrix}
				 x_{k-N}  \\[0.21em] 0 \\\vdots  \\[0.21em]  0  \end{bmatrix}.
 \end{align} 
  With $\mathbf{\hat{z}}^{ i}_{k} $, we denote the $i$-th iterate of the optimization algorithm at time $k$.
With $ \mathbf{\bar{z}}_k $, we refer to the a priori estimate at time $k$ and with   $ \mathbf{z}_k$ 	 to the true state $x_{k-N}$ with true model residual sequence    $  \lbrace  0  , \cdots , 0  \rbrace $. 

Upon arrival of a new measurement at time $k$, the optimization algorithm is initialized based on the a priori estimate~$\mathbf{\bar{z}}_k$.  In particular, we compute the warm start $\mathbf{\hat{z}}^{0}_{k}$ as the Bregman projection  of $\mathbf{\bar{z}}_k$ onto the constraint set $ \mathcal{S}_k$ as formulated in line~3 of Algorithm \ref{alg: eta changes}.
Then, a fixed number $\mathrm{it}(k)$ of optimization iterations is performed via (\ref{optAdap}), generating $\big\lbrace  \mathbf{\hat{z}}^{1}_{k}, \cdots,\mathbf{\hat{z}}^{\mathrm{it}(k)}_{k} \big\rbrace $. Here, $\eta_k^i >0$ denotes the step size employed at the $i$-th iteration at time $k$. From this sequence of iterates, an arbitrary iterate  $\mathbf{\hat{z}}^{\mathrm{j}(k)}_{k} $, $\mathrm{j}(k) \in \lbrace 0,\cdots,\mathrm{it}(k)\rbrace$ can be chosen, based on which the state estimate $\hat{x}_k$ is obtained using 
\begin{align}
		\label{StateEst}
\hat{x}_k =A^N\, \hat{x}^{\mathrm{j}(k)}_{k-N}+ \!\!\sum_{j=k-N} ^ {k-1}\! \! A^{k-1-j}\, \left( B\,  u_j+ \hat{w}^{\mathrm{j}(k)}_{j} \right)
\end{align}
for $k>N$ (see Remark \ref{rem: indices} for the case where $0< k \leq N$). Moreover, the a priori estimate $ \mathbf{\bar{z}}_{k+1} $ for the next time instant is computed through the operator $\Phi_k : \mathbb{R}^{(N+1)n} \rightarrow \mathbb{R}^{(N+1)n}$ which will be defined in~\eqref{Adap: warm}.  As mentioned above, the basic idea of the  pMHE framework is to use the Bregman distance~$D_\psi$ as a proximity measure to a  stabilizing a priori estimate in order to inherit its  stability properties. Since the Luenberger observer  appears as a simple candidate for constructing the a priori estimates,  we require that the operator $\Phi_k$ incorporates its dynamics as follows:
			    \begin{align}
			     \label{Adap: warm}
\Phi_k \left( \mathbf{\hat{z}}^{ \mathrm{j}(k)}_{k} \right) \! \coloneqq  \!  \begin{bmatrix}
   \! A\,  \hat{x}_{k-N}^{ \mathrm{j}(k)} +\!B\, u_{k-N}+\!L \left(y_{k-N} - C  \hat{x}_{k-N}^{ \mathrm{j}(k)} \right)     \\\mathbf{0}
     \end{bmatrix}\!,
       \end{align}	
where $\mathbf{0}\in \mathbb{R}^{Nn}$. Here, the observer gain $L$ is chosen such that all the eigenvalues of $A-LC$ are strictly within the unit circle. 
 \begin{algorithm} [t]
\caption{Anytime pMHE}
\label{alg: eta changes}
\begin{algorithmic}[1]
\State  \textbf{Initialize:}   Choose  $ \hat{x}_0  $ and set $ \bar{\mathbf{z}}_0  =\hat{x}_0$ 
  \For {$k=1,2, \cdots$  }
    \State  $ \mathbf{\hat{z}}_k^0= \argmin\limits_{\quad\enskip\mathbf{z}\in  \mathcal{S}_k} \quad  D_\psi (\mathbf{z},\mathbf{\bar{z}}_k)$  warm start
      \For {$i= 0,\dots, \mathrm{it}(k)-1$ } optimizer update
      \begin{align}
  \label{optAdap}
\hspace*{-0.2cm}  \mathbf{\hat{z}}^{i+1}_{k}=\argmin_{\quad\enskip\mathbf{z}\in  \mathcal{S}_k}\left\lbrace  \eta_k^i\,    \nabla f_k\left(  \mathbf{\hat{z}}^{i}_{k} \right)^\top \mathbf{z}+ D_\psi (\mathbf{z},  \mathbf{\hat{z}}^{i}_{k}) \right\rbrace   
   \end{align}
		\EndFor
		\State for some $\mathrm{j}(k) \in \lbrace 0,\cdots,\mathrm{it}(k)\rbrace$ obtain $\hat{x}_k$ from (\ref{StateEst})
		\State  $\mathbf{\bar{z}}_{k+1} =\Phi_k \left( \mathbf{\hat{z}}^{\mathrm{j}(k)}_{k} \right)$  
		\EndFor
\end{algorithmic}
\end{algorithm}
In the following, we compare Algorithm~\ref{alg: eta changes} with our earlier formulation of pMHE, given in  Algorithm~\ref{alg: optimal pMHE}.  
 Observe that, while a solution of the optimization problem  (\ref{condensedpMHE}) is computed,  step~4 in Algorithm \ref{alg: eta changes} employs the so-called  mirror descent algorithm~\cite{beck2003mirror} that iterates~\eqref{optAdap}  until  a given number of iterations $\mathrm{it}(k)$ is achieved. For $D_\psi \left( \mathbf{z}_1, \mathbf{z}_2 \right) = \frac{1}{2} \left\Vert \mathbf{z}_1 -  \mathbf{z}_2 \right\Vert^2$ and $ \mathcal{S}_k=\mathbb{R}^{(N+1)n}$, the optimizer update step (\ref{optAdap}) corresponds to an iteration step of the classical gradient descent algorithm and hence step 4 
 can be executed very quickly. 
In the constrained case, (\ref{optAdap}) can be regarded as a generalization of the projected gradient algorithm \cite{beck2003mirror}.  For this reason, we can view Algorithm  \ref{alg: eta changes} as a real-time version of the pMHE scheme given  in  Algorithm~\ref{alg: optimal pMHE}. \rev{Note also that choosing the so-called Kullback-Leiber divergence  as Bregman distance yields the efficient entropic descent algorithm if the constraint set is given by the unit simplex  \cite{beck2003mirror}.} An appealing feature of  Algorithm \ref{alg: eta changes} is that, depending on the available computation time between two subsequent time instants $k$ and $k+1$, the user can specify a maximum number of iterations $\mathrm{it}(k)$ after which   the optimization algorithm at time $k$ has to return a solution. \\
Another key difference between the two algorithms is that Algorithm \ref{alg: optimal pMHE}  is biased by the stabilizing a priori estimate $\bar{\mathbf{z}}_k$, while this bias  is fading away  in  Algorithm \ref{alg: eta changes}. In other words, the a priori estimate constructed based on the Luenberger observer (\ref{Adap: warm}) has less impact at each optimization iteration, which improves the performance of the pMHE iteration scheme.  
This is due to the fact that $\bar{\mathbf{z}}_k$
might  degenerate  performance in Algorithm \ref{alg: optimal pMHE}, since
the  solution lies in proximity to the a priori estimate.
In Algorithm \ref{alg: eta changes}, however, the Luenberger observer enters
only in the warm start. From this point of view,
it is quite surprising that, even though the effect of this stabilizing ingredient is fading away, stability is provably preserved, as we will show in the subsequent section. Thus, this "implicit stabilizing regularization" 
approach of the a priori estimate 
is in contrast to the explicit stabilizing regularization proposed in \cite{sui2014linear,gharbi2018proximity} (see also \cite{belabbas2020implicit}).
Moreover, the proposed MHE algorithm possesses the anytime property.
  The \emph{anytime} property refers to the fact that the algorithm will yield stable estimation errors after any number of  optimization algorithm iterations. This is similar in spirit to anytime model predictive control (MPC) algorithms, which compute stabilizing  control inputs after any  optimization iteration \cite{feller2017stabilizing,feller2018sparsity}.
\begin{rem} \label{rem: indices} For $k \leq N$, we can employ the steps of Algorithm \ref{alg: eta changes} by setting all the negative indices to zero. More specifically,   
$  \mathbf{\hat{z}}^{ i}_{k} =\begin{bmatrix} (\hat{x}^{ i}_{0})^\top &  (\hat{w}_{0}^{ i})^\top & \hdots  & (\hat{w}_{k-1}^{ i})^\top    \end{bmatrix}^\top$ and  $\mathbf{\bar{z}}_{k}  =\begin{bmatrix} \hat{x}_{0}^\top &  0 & \hdots  & 0 \end{bmatrix}^\top$ where $\mathbf{\hat{z}}^{ i}_{k},\mathbf{\bar{z}}_{k}   \in \mathbf{R}^{(k+1)n} $. In order to compute the state estimate, \eqref{StateEst} has to be explicitly modified to
 \begin{align}
\hat{x}_k =A^k\, \hat{x}^{ \mathrm{j}(k)}_{0}+  \sum_{j=0} ^ {k-1}  A^{k-1-j}\,\left( B\,   u_j + \hat{w}^{ \mathrm{j}(k)}_{j} \right).
\end{align}
 \end{rem}
We impose the following assumptions. 
\begin{ass}[Properties of $ \mathcal{S}_k$] \label{ass: S closed}
 The  set $ \mathcal{S}_k$ of constraints is closed and convex with nonempty interior.
\end{ass}
  \begin{ass}[Convexity of $f_k$] \label{ass: f convex}
The sum of stage costs  $f_k$ is continuously differentiable,  convex for all $k> 0$, and achieves its minimum at $\mathbf{z}_k$. 
   \end{ass}
     \begin{ass}[Strong smoothness of $f_k$] \label{ass: f strngSmooth}
The sum of stage costs  $f_k$ is strongly smooth with constant $L_f>0 $:
\begin{align}
 f_k(\mathbf{z}_2) \leq  f_k(\mathbf{z}_1)  +\nabla f_k(\mathbf{z}_1)^\top   ( \mathbf{z}_2 -\mathbf{z}_1)  + \frac{L_f}{2} \Vert  \mathbf{z}_1 - \mathbf{z}_2  \Vert^2 
\end{align}
for all $\mathbf{z}_1, \mathbf{z}_2 \in \mathbb{R}^{(N+1)n}  $ and $k>0$.
   \end{ass}
 \begin{ass} [Strong convexity and smoothness of $D_\psi$] \label{ass: Bregman}  The function $\psi$ is continuously differentiable, strongly convex  with constant $\sigma>0$ and strongly smooth with constant  $\gamma>0$, which implies the following for the Bregman distance  
 \begin{align}
 \frac{\sigma}{2 } \Vert  \mathbf{z}_1 - \mathbf{z}_2  \Vert^2  \leq   D_\psi (\mathbf{z}_1, \mathbf{z}_2 ) \leq  \frac{\gamma}{2} \Vert  \mathbf{z}_1 - \mathbf{z}_2  \Vert^2 
 \end{align}
 for all $\mathbf{z}_1, \mathbf{z}_2 \in \mathbb{R}^{(N+1)n}  $.
   \end{ass}
In Assumption \ref{ass: f convex}, we can ensure that $f_k$ achieves its minimum at $\mathbf{z}_k$ by designing the stage costs $r(\cdot)$ and $q(\cdot)$ such that their corresponding minimum is achieved at zero. \rev{Requiring strong smoothness of $f_k$ in Assumption \ref{ass: f strngSmooth} is	rather customary in the analysis of first-order optimization algorithms \cite{mokhtari2016online} and will prove central in the subsequent theoretical studies of the pMHE iteration scheme. Notice that Assumption \ref{ass: Bregman} restricts the  class of employed  Bregman distances to functions which can be quadratically lower and upper bounded. Obviously, this includes the important special case of quadratic distances which are widely used as prior weighting in the MHE literature in order to ensure stability of the estimation error. In addition,we can use any Bregman distance $D_{\psi}$ constructed based on the function  $\psi(\mathbf{z})= \frac{1}{2}\Vert \mathbf{z} \Vert_{P}^2 + B(\mathbf{z})$, where $B:\mathbb{R}^{(N+1)n} \rightarrow \mathbb{R}$ is a convex and strongly smooth function, i.e., a convex function whose gradient is Lipschitz continuous. }

%% file: figures/iteration_scheme_sketch.tex
\begin{tikzpicture}[scale=0.48,domain=-2:8.5]


\fill[draw=gray!20,color=gray!20] (0,2) rectangle (0.875, 8) node[pos=0.22, xshift=1mm ,font=\scriptsize ,rotate=90,
] {\color{black}warm start};
\fill[draw=gray!20,color=gray!20] (0,1) rectangle (0.875, 1.15);
\fill[draw=gray!20,color=gray!20] (-7,2) rectangle (-6.125, 8) node[pos=0.22, xshift=1mm ,font=\scriptsize ,rotate=90,
] {};
\fill[draw=gray!20,color=gray!20] (-7,1) rectangle (-6.125, 1.15);
\fill[draw=gray!20,color=gray!20] (7,2) rectangle (7.875, 8) node[pos=0.22, xshift=1mm ,font=\scriptsize ,rotate=90,
] {};
\fill[draw=gray!20,color=gray!20] (7,1) rectangle (7.875, 1.15);

\draw[->,latex-latex] (0,8) --   node[above] {\footnotesize{$k$  iteration phase}}  (7,8);
\draw[->,latex-latex] (-7,8) --   node[above] {\footnotesize{$k\!-\!1$ iteration phase}}  (0,8);
\draw[->,latex- ] (7,8) --   node[above] {\footnotesize{$\quad \cdots$}}  (8,8);
\draw[->,-latex] (-8,8) --   node[above] {}  (-7,8);

\draw[->,-latex] (-8.7,1) -- (9,1);
\foreach \x in {-7,0,7}
    \draw (\x,1.15) -- (\x,0.8 );
\foreach \x in {0.875, 2.625, 4.375, 6.125}
   { \draw (\x,1.15) -- (\x,1 );
    \draw (\x,7.85) -- (\x,8 ) ; } ;     
    \foreach \x in {-6.125, -4.8125, -3.5, -2.1875,-0.875}
   { \draw (\x,1.15) -- (\x,1 );
    \draw (\x,7.85) -- (\x,8 ) ; } ; 
\draw (-7,0.9) node[anchor=north,font=\footnotesize] {$k-1$};
\draw (-7,2.0 ) node[anchor=north,font=\footnotesize] {meas. $y_{k-1}$};
\draw (0,0.9) node[anchor=north,font=\footnotesize] {$k$};
\draw (0,2.0 ) node[anchor=north,font=\footnotesize] {meas. $y_k$};
\draw (7,0.9) node[anchor=north,font=\footnotesize] {$k+1$};
\draw (7,2.0 ) node[anchor=north,font=\footnotesize] {meas. $y_{k+1}$};

\draw (0.875,7.9) node[anchor=north,font=\tiny] {$i\!=\!0$};
\draw (2.625,7.9) node[anchor=north,font=\tiny ] {$i\!=\!1$};
\draw ( 4.375,7.9) node[anchor=north,font=\tiny ] {$i\!=\!2$};
 \draw (6.125,7.9) node[anchor=north,font=\tiny ] {$\mathrm{it}(\!k\!)\!=\!3$};
 \draw (-6.125,7.9) node[anchor=north,font=\tiny] {$i\!=\!0$};
\draw (-4.8125,7.9) node[anchor=north,font=\tiny ] {$i\!=\!1$};
\draw (-0.875,7.9) node[anchor=north,font=\tiny ] {$\mathrm{it}(\!k\!-\!1\!)\!$};
\draw (-0.875,7.4) node[anchor=north,font=\tiny ] {$=\!4$};
\draw (-3.5,7.9) node[anchor=north,font=\tiny ] {$\cdots$};

\draw[-,dashed] (-7,2) -- (-7,9)    ;
\draw[-,dashed] (0,2) -- (0,9)    ;
\draw[-,dashed] (7,2) -- (7,9)    ;

\draw[color=black ] plot [ mark=*, mark options={fill=white},mark repeat=2 ] coordinates { 
(0.875,5.8) (2.625,5.8) (2.625,4.8) (4.375,4.8) (4.375,5.5) (6.125,5.5) (6.125,5.9)  (7,5.9)   };
\draw[color=black ] plot [ mark=*, mark options={fill=gray}  ] coordinates { 
 (7,5.9)    };
 \draw[color=black ] plot [ mark=*, mark options={fill=black}  ] coordinates { 
 (7,4.3)    };
\draw[color=black] (0.875,5.9) node[anchor=south,font=\scriptsize] {$\mathbf{\hat{z}}_k^0$};
\draw[color=black] (6.3,5.9) node[anchor=south,font=\scriptsize] {$\mathbf{\hat{z}}_k^{3}$};
\draw[color=black] (7.75,4.2) node[anchor=south,font=\scriptsize] {$\mathbf{\bar{z}}_{k+1} $};

\draw[color=black ] plot [ mark=*, mark options={fill=white},mark repeat=2 ] coordinates { 
(-6.125,3.8) (-4.8125,3.8) (-4.8125,4.8) (-3.5,4.8) (-3.5,5.5) (-2.1875,5.5) (-2.1875,4.5)  (-0.875,4.5) (-0.875,3.3) (0,3.3)   };
 \draw[color=black ] plot [ mark=*, mark options={fill=gray}  ] coordinates { 
 (0,3.3)    };
 \draw[color=black ] plot [ mark=*, mark options={fill=black}  ] coordinates { 
 (0,5)     };
  \draw[color=black ] plot [ mark=*, mark options={fill=black}  ] coordinates { 
  (-7, 6)    };
\draw[color=black] (-6.125,3.9) node[anchor=south,font=\scriptsize] {$\mathbf{\hat{z}}_{k-1}^0$};
\draw[color=black] (-0.78,3.3) node[anchor=north,font=\scriptsize] {$\mathbf{\hat{z}}_{k-1}^{4}$};
\draw[color=black] (0.45,4.8) node[anchor=south,font=\scriptsize] {$\mathbf{\bar{z}}_{k} $};
\draw[color=black] (-7.7,5.9) node[anchor=south,font=\scriptsize] {$\mathbf{\bar{z}}_{k-1} $};

\end{tikzpicture}

%% file: stability.tex
\section{Stability Analysis}
\label{sec: stability}
In this section, we analyze the stability properties of the proposed pMHE iteration scheme (Algorithm \ref{alg: eta changes}). 
More specifically, we derive sufficient conditions on the Bregman distance~$D_\psi$ and on the step sizes $\eta_k^i$  for the global exponential stability (GES) of the estimation error 
\begin{align}
\label{EstimError}
 e_{k-N} \coloneqq x_{k-N} - \hat{x}_{k-N}^{j}  
\end{align}
for any $j \in \{0,...,\mathrm{it}(k)\}.$
 The following key result establishes the stability properties of the pMHE iteration scheme.
 \begin{theorem}
\label{prop: Stabadap}
 Consider  Algorithm \ref{alg: eta changes}    and suppose that Assumptions~\ref{ass: S closed}-\ref{ass: Bregman} hold. If we choose the Bregman distance $D_\psi$ such that
 \begin{align}
 \label{DiffDBoundadap}
D_\psi\left(\Phi_k(\mathbf{z}), \Phi_k(\mathbf{\hat{z}})\right) - D_\psi(\mathbf{z},\mathbf{\hat{z}})  \leq  - c \, \Vert \mathbf{z} - \mathbf{\hat{z}}   \Vert^2,
\end{align}
is satisfied for all $ \mathbf{z},\mathbf{\hat{z}} \in \mathbb{R}^{(N+1)n}$, where $ \Phi_k(\cdot) $ is defined in~\eqref{Adap: warm}, and if the step size at the $i$-th iteration and time instant $k$ satisfies 
\begin{align}
\label{stepSizeadap}
\eta_k^i \leq \frac{ \sigma }{L_f},
\end{align}
then  the estimation error \eqref{EstimError} is GES.
\end{theorem}

To prove this theorem, we require the following result.
  \begin{lem}
  \label{lem: Stabadap}
   Consider  Algorithm \ref{alg: eta changes}  and suppose Assumptions  \ref{ass: S closed}-\ref{ass: Bregman} hold true.   Then, for two consecutive iterates  $\mathbf{\hat{z}}_{k}^{i}$ and $\mathbf{\hat{z}}_{k}^{i+1}$
   at any time instant $k>0$, we obtain  
           \begin{align}
        \label{result: Stabadap}
D_\psi \big(\mathbf{z}_k, \mathbf{\hat{z}}_{k}^{i+1}  \big) & \leq D_\psi \big(\mathbf{z}_k,\mathbf{\hat{z}}_{k}^i\big) +      \frac{1}{2}  \big(\eta_k^i\, L_f -  \sigma \big) \left\Vert \mathbf{\hat{z}}_{k}^{i+1}   - \mathbf{\hat{z}}_{k}^i \right\Vert^2,
   \end{align} 
  where $i\in \{0,...,\mathrm{it}(k)\}$ and $\mathbf{z}_k $ is defined in \eqref{notation}. Moreover,  we have that for any $j\in\{0,...,\mathrm{it}(k)\}$
  \begin{align}
  \label{result: StabadapiT}  
D_\psi\big(\mathbf{z}_k, \mathbf{\hat{z}}_{k}^{j}\big) &\leq D_\psi\big(\mathbf{z}_k,  \mathbf{\bar{z}}_k \big)   \\& + \frac{1 }{2}  \sum_{i=0}^{j-1} \big(\eta_k^{i} \,L_f - \sigma  \big)\left\Vert  \mathbf{\hat{z}}_{k}^{i+1} -\mathbf{\hat{z}}_{k}^{i}   \right\Vert^2. \nonumber 
\end{align}
    \end{lem}
 The proof of Lemma~\ref{lem: Stabadap} can be found in  Appendix \ref{app:stabilityLemma}. We are now in a position to prove the stability result for the proposed pMHE scheme.

\begin{proof}[Proof of Theorem~\ref{prop: Stabadap}]
We first prove GES of the estimation error (\ref{EstimError}) with $j=\mathrm{j}(k)$. Let $V$ be a candidate Lyapunov function chosen as the Bregman distance  in (\ref{optAdap}), i.e.,
\begin{align}
 V\big(  \mathbf{z}_{k}, \mathbf{\hat{z}}^{\mathrm{j}(k)}_{k}\big)  = D_\psi\big( \mathbf{z}_{k},\mathbf{\hat{z}}^{\mathrm{j}(k)}_{k}\big).
\end{align}
Here, $\mathbf{z}_k$ denotes the true state with zero model residual as defined in (\ref{notation}) and $\mathbf{\hat{z}}^{\mathrm{j}(k)}_{k}$ the  selected  pMHE iterate   at time instant $k$. In the following, we show that $V $ satisfies the following conditions
\begin{subequations}
\begin{align}
\label{firstV}
\alpha_1 \big\Vert \mathbf{z}_{k}-\mathbf{\hat{z}}^{ \mathrm{j}(k)}_{k} \big\Vert^2 \leq  V\big(\mathbf{z}_{k},\mathbf{\hat{z}}^{ \mathrm{j}(k)}_{k}\big) \leq \alpha_2 \big\Vert \mathbf{z}_{k}-\mathbf{\hat{z}}^{\mathrm{j}(k)}_{k} \big\Vert^2 
\end{align}
and 
\begin{align}
\label{secondV}
\Delta V &\coloneqq V\big( \mathbf{z}_{k+1},\mathbf{\hat{z}}^{ \mathrm{j}(k+1)}_{k+1}\big) - V\big(  \mathbf{z}_{k},\mathbf{\hat{z}}^{ \mathrm{j}(k)}_{k}\big)   \\& \leq -\alpha_3 \big\Vert  \mathbf{z}_{k}-\mathbf{\hat{z}}^{\mathrm{j}(k)}_{k} \big\Vert^2  \nonumber
\end{align}
\end{subequations}
for some positive constants $\alpha_1$, $\alpha_2$ and $\alpha_3$. Note that, in view of (\ref{notation}), the error generated by the pMHE iteration scheme at time $k$ is given by 
\begin{align}
\mathbf{z}_{k}-\mathbf{\hat{z}}^{\mathrm{j}(k)}_{k}= \begin{bmatrix}
\vspace{0.1cm} x_{k-N}  \\[0.24em] \mathbf{0}_{[Nn]}  
\end{bmatrix} -  \begin{bmatrix}
 \hat{x}_{k-N}^{\mathrm{j}(k)}   \\[0.2em]  \mathbf{\hat{w}}^{ \mathrm{j}(k)}_{k}  
\end{bmatrix} =  \begin{bmatrix} e_{k-N}  \\[0.2em] -\mathbf{\hat{w}}^{ \mathrm{j}(k)}_{k} \end{bmatrix},
\end{align}
where $ \mathbf{\hat{w}}^{ \mathrm{j}(k)}_{k} \coloneqq \begin{bmatrix} \big(\hat{w}_{k-N}^{\mathrm{j}(k)}\big)^\top & \cdots & \big(\hat{w}_{k-1}^{\mathrm{j}(k)}\big)^\top   \end{bmatrix}^\top$.
By Assumption \ref{ass: Bregman},  (\ref{firstV}) follows with $\alpha_1 = \frac{\sigma}{2}$ and  $\alpha_2 = \frac{\gamma}{2}$. Furthermore, by (\ref{result: StabadapiT}) in Lemma \ref{lem: Stabadap}, we have
 \begin{align}
 \label{Vdiff1}
 \Delta V &=  D_\psi\big( \mathbf{z}_{k+1},\mathbf{\hat{z}}^{ \mathrm{j}(k+1)}_{k+1}\big)  - D_\psi\big( \mathbf{z}_{k},\mathbf{\hat{z}}^{ \mathrm{j}(k)}_{k}\big)  \\ 
& \leq D_\psi(\mathbf{z}_{k+1},  \mathbf{\bar{z}}_{k+1} )   - D_\psi\big( \mathbf{z}_{k},\mathbf{\hat{z}}^{ \mathrm{j}(k)}_{k}\big) \nonumber\\& \quad+ \frac{1 }{2}  \sum_{i=0}^{\mathrm{j}(k+1)-1} \big(\eta_{k+1}^{i} \,L_f - \sigma  \big)\left\Vert  \mathbf{\hat{z}}_{k+1}^{i+1} -\mathbf{\hat{z}}_{k+1}^{i}   \right\Vert^2. \nonumber    
\end{align}
The condition on the step sizes given in (\ref{stepSizeadap}) implies that $
\eta_{k+1}^{i} \,L_f - \sigma \leq 0$ and hence,
\begin{align}
\label{VdiffsumNeg}
\frac{1 }{2}  \sum_{i=0}^{\mathrm{j}(k+1)-1} \big(\eta_{k+1}^{i} \,L_f - \sigma  \big)\left\Vert   \mathbf{\hat{z}}_{k+1}^{i+1} -\mathbf{\hat{z}}_{k+1}^{i} \right\Vert^2 \leq 0.
\end{align}
Moreover, since $ \mathbf{\bar{z}}_{k+1} = \Phi_k\big( \mathbf{\hat{z}}^{ \mathrm{j}(k)}_{k}\big)$ and  
	\begin{align}
\Phi_k \left(\mathbf{z}_{k} \right) \! &=  \!  \begin{bmatrix}
   \! A\,  x_{k-N} +\!B\, u_{k-N}+\!L \left(y_{k-N} - C  x_{k-N}  \right)     \\\mathbf{0}_{[Nn]}  
     \end{bmatrix}\nonumber \\ 
     &=  \!  \begin{bmatrix}
   \! A\,  x_{k-N} +\!B\, u_{k-N}     \\\mathbf{0}_{[Nn]}  
     \end{bmatrix} = \mathbf{z}_{k+1} 
       \end{align}
in view of (\ref{Adap: warm}), we have 
 \begin{align}
 \label{diff21}
\Delta V 
&\leq    D_\psi(\mathbf{z}_{k+1},  \mathbf{\bar{z}}_{k+1} )   - D_\psi\big( \mathbf{z}_{k},\mathbf{\hat{z}}^{ \mathrm{j}(k)}_{k}\big)  \\
&= D_\psi\big(  \Phi_k\big(\mathbf{z}_{k} \big), \Phi_k\big( \mathbf{\hat{z}}^{ \mathrm{j}(k)}_{k}\big)\big)-D_\psi\big( \mathbf{z}_{k},\mathbf{\hat{z}}^{ \mathrm{j}(k)}_{k}\big)   \nonumber.
\end{align}

 Given that the Bregman distance satisfies  (\ref{DiffDBoundadap}), we obtain \begin{align}
 \label{Vdiff2}
\Delta V 
&\leq - c \,\big\Vert \mathbf{z}_{k} -  \mathbf{\hat{z}}^{ \mathrm{j}(k)}_{k} \big\Vert^2  \\
& =  - c \, \big\Vert  e_{k-N} \big\Vert^2  -c\, \big\Vert  \mathbf{\hat{w}}^{ \mathrm{j}(k)}_{k} \big\Vert^2\nonumber.
\end{align}
Hence, the candidate Lyapunov function satisfies (\ref{secondV}) with $\alpha_3 = c$ and the estimation error (\ref{EstimError}) with $j=\mathrm{j}(k)$ is GES. In the following, we show that GES holds also for any $j~\in~\{0,...,\mathrm{it}(k)\}$. Based on the previous Lyapunov analysis, we have that 
 \begin{align}
 & D_\psi \big( \mathbf{z}_{k} ,\mathbf{\hat{z}}^{\mathrm{j}(k )}_{k} \big)  - D_\psi\big( \mathbf{z}_{k-1},\mathbf{\hat{z}}^{ \mathrm{j}(k-1)}_{k-1} \big) \\
&\quad\leq   -  c \Vert \mathbf{z}_{k-1} - \mathbf{\hat{z}}^{ \mathrm{j}(k-1)}_{k-1} \Vert^2  \leq  -\frac{2c}{\gamma} D_\psi\big( \mathbf{z}_{k-1},\mathbf{\hat{z}}^{ \mathrm{j}(k-1)}_{k-1}\big) , \nonumber
\end{align} 
where the last inequality holds by the strong smoothness of the Bregman distance.
By defining $\beta_e \coloneqq 1 - \frac{2c}{\gamma}$, 
\begin{align}
0 \leq  D_\psi \big( \mathbf{z}_{k} ,\mathbf{\hat{z}}^{\mathrm{j}(k )}_{k} \big)  \leq \beta_e \, D_\psi\big( \mathbf{z}_{k-1},\mathbf{\hat{z}}^{ \mathrm{j}(k-1)}_{k-1}\big),
\end{align}
 where $\beta_e \in [0,1)$ since $D_\psi$ is nonnegative and $\frac{2c}{\gamma}>0$. Hence
\begin{align}
\label{toreplace}
 D_\psi\big(    \mathbf{z}_{k},\mathbf{\hat{z}}^{ \mathrm{j}(k)}_{k}  \big) \leq \beta_e^{k} \, D_\psi\big( \mathbf{z}_{0},\mathbf{\bar{z}}_0\big).
\end{align}
We consider the difference 
$   D_\psi\big( \mathbf{z}_{k+1},\mathbf{\hat{z}}^{ j}_{k+1}\big)  - D_\psi\big( \mathbf{z}_{k},\mathbf{\hat{z}}^{ \mathrm{j}(k)}_{k}\big)
$
for any $j\in\{0,...,\mathrm{it}(k+1)\}$. By (\ref{result: StabadapiT}), we have  
  \begin{align}
 &D_\psi\big( \mathbf{z}_{k+1},\mathbf{\hat{z}}^{ j}_{k+1}\big)  - D_\psi\big( \mathbf{z}_{k},\mathbf{\hat{z}}^{ \mathrm{j}(k)}_{k}\big) \\ &\quad\leq  D_\psi\big( \mathbf{z}_{k+1},\mathbf{\bar{z}}_{k+1}\big)  - D_\psi\big( \mathbf{z}_{k},\mathbf{\hat{z}}^{ \mathrm{j}(k)}_{k}\big)\nonumber \\
 &\quad\leq - c \,\big\Vert \mathbf{z}_{k} -  \mathbf{\hat{z}}^{ \mathrm{j}(k)}_{k} \big\Vert^2 \nonumber,
\end{align}
where the last inequality holds in view of (\ref{diff21}) and (\ref{Vdiff2}). Hence, by using $\beta_e$ again, we have for any $j\in\{0,...,\mathrm{it}(k+1)\}$
\begin{align}
  D_\psi \big( \mathbf{z}_{k+1} ,\mathbf{\hat{z}}^{j}_{k+1} \big)  \leq \beta_e \, D_\psi\big( \mathbf{z}_{k},\mathbf{\hat{z}}^{ \mathrm{j}(k)}_{k}\big).
\end{align} 
By (\ref{toreplace}), we obtain  $\forall k>0$
\begin{align}
  D_\psi \big( \mathbf{z}_{k+1} ,\mathbf{\hat{z}}^{j}_{k+1} \big)  \leq \beta_e^{k+1} \, D_\psi\big( \mathbf{z}_{0},\mathbf{\bar{z}}_0\big).
\end{align}
By the strong smoothness and convexity of $D_\psi$, we therefore get the following GES property of the estimation error 
\begin{align}
\label{GESproperty1}
  \big\Vert \mathbf{z}_k - \mathbf{\hat{z}}_{k}^{j} \big\Vert^2   &\leq    \frac{\gamma}{\sigma}\beta_e^{k} \,  \big\Vert \mathbf{z}_{0}-\mathbf{\bar{z}}_0\big\Vert^2.
\end{align}
\end{proof}
The theorem implies that stability of the estimation error is guaranteed 
for any iterate and is
independent  of which iterate $\mathbf{\hat{z}}^{ \mathrm{j}(k)}_{k}$ is picked from the sequence  $\big\lbrace  \mathbf{\hat{z}}^{0}_{k}, \cdots,\mathbf{\hat{z}}^{\mathrm{it}(k)}_{k} \big\rbrace $
in Step 7 in Algorithm \ref{alg: eta changes}, as well as of the number of iterations $\mathrm{it}(k)$. Hence, the algorithm generates convergent estimates after each optimizer update step and can be therefore considered as an  anytime  MHE algorithm.   Moreover, selecting $\eta_k^i = \frac{\sigma}{L_f}$, for all $i$ and $k$, is sufficient for ensuring GES. \rev{In addition, this guarantee holds independently of the choice of the horizon length $N\in \mathbb{N}_+$ and for any convex stage cost satisfying Assumptions \ref{ass: f convex} and \ref{ass: f strngSmooth}. This includes for instance quadratic functions and the Huber penalty function,  which allows to handle the important case of measurement outliers.} 

 It is worth pointing out that detectability of the pair $(A,C)$ implies that we can find  suitable choices of the Bregman distance $D_\psi$  that fulfill condition (\ref{DiffDBoundadap}). In particular, let
\begin{align}
\label{quadBregm}
D_\psi(\mathbf{z}_1,\mathbf{z}_2)  = \frac{1}{2}  \Vert x_1 -x_2 \Vert^2_P + \frac{1}{2}  \Vert \mathbf{w}_1 -\mathbf{w}_2 \Vert^2_W
\end{align}
with $\mathbf{z}_1= \begin{bmatrix}
 x_1^\top & \mathbf{w}_1^\top  \end{bmatrix}^\top$,  $\mathbf{z}_2=\begin{bmatrix}
 x_2^\top & \mathbf{w}_2^\top  \end{bmatrix}^\top$, $x_1, x_2 \in \mathbb{R}^n$, $\mathbf{w}_1, \mathbf{w}_2 \in \mathbb{R}^{Nn}$, and $P \in \mathbb{S}_{+\!\!+}^n$, $W \in \mathbb{S}_{+\!\!+}^{Nn}$. Using (\ref{Adap: warm}) and a simple algebraic manipulation, we have that
 \begin{align}
 \label{Bregman difference }
 &D_\psi(\Phi_k(\mathbf{z}_1), \Phi_k(\mathbf{z}_2)) - D_\psi(\mathbf{z}_1,\mathbf{z}_2) \\ & \qquad=     \frac{1}{2} \left\Vert (A-LC) (x_1 - x_2)  \right\Vert^2_P - \frac{1}{2} \left\Vert x_1 - x_2   \right\Vert^2_P \nonumber \\ &\qquad \qquad+ \frac{1}{2} \left\Vert \mathbf{0} \right\Vert^2_W - \frac{1}{2} \left\Vert \mathbf{w}_1 - \mathbf{w}_2 \right\Vert^2_W \nonumber.
 \end{align}
 Hence, satisfying (\ref{DiffDBoundadap}) amounts to designing  the weight matrix $P \in \mathbb{S}_{+\!\!+}^n$ such that the linear matrix inequality  (LMI) 
\begin{align}
\label{LMIBregman}
(A-LC )^ \top P \, (A-LC) -P \prec - Q  
\end{align} 
holds for some $Q \in \mathbb{S}_{+\!\!+}^n$. This is because when (\ref{LMIBregman}) holds,~(\ref{Bregman difference }) yields 
 \begin{align}
 &D_\psi(\Phi_k(\mathbf{z}_1), \Phi_k(\mathbf{z}_2)) - D_\psi(\mathbf{z}_1,\mathbf{z}_2) \\ & \qquad\leq  - \frac{\lambda_{\min}(Q)}{2} \left\Vert x_1 - x_2   \right\Vert^2 - \frac{\lambda_{\min}(W)}{2} \left\Vert \mathbf{w}_1 - \mathbf{w}_2 \right\Vert^2 \nonumber \\
 & \qquad \leq -c \, \Vert \mathbf{z}_1 - \mathbf{z}_2 \Vert^2  \nonumber
 \end{align}
 where $c = \frac{1}{2} \min\lbrace \lambda_{\min}(Q), \lambda_{\min}(W) \rbrace $.

%% file: regret.tex
\section{Regret Analysis}
\label{sec: regret}
 In this section, we study the performance of the proposed anytime pMHE iteration scheme. Recall the  performance  criterion  of the original estimation problem (\ref{condensed}), which   is  to minimize at each time instant $k$ the sum of stage costs  $f_k$.   In order to characterize the overall performance of  Algorithm~\ref{alg: eta changes},  \rev{we investigate the accumulation of losses $f_k$ over 	the considered simulation time $T\in \mathbb{N}_+$ given by}
\begin{align} 
\label{accCost}
\sum_{k=1}^{T}\quad \min_{0 \leq i \leq \mathrm{it}(k) } f_k (   \mathbf{\hat{z}}^{i}_{k} ) . 
\end{align}
Note that the $\min$ operator in (\ref{accCost}) follows from the fact that the generated sequence  of iterates  $\lbrace \mathbf{\hat{z}}^{0}_{k},\cdots,  \mathbf{\hat{z}}^{ \mathrm{it}(k)}_{k}  \rbrace$ does not  necessarily  produce $f_k(\mathbf{\hat{z}}^{0}_{k}) \geq  \cdots \geq f_k(\mathbf{\hat{z}}^{ \mathrm{it}(k)}_{k}) $. \rev{Hence, given $\lbrace \mathbf{\hat{z}}^{0}_{k},\cdots,  \mathbf{\hat{z}}^{ \mathrm{it}(k)}_{k}  \rbrace$, we have to choose a suitable  $\mathbf{\hat{z}}^{i}_{k}$ whose function value is then used in the performance analysis. Following the literature on mirror descent algorithms \cite{beck2003mirror},} we select the iterate  with the minimal cost as our estimate, i.e., $\mathbf{\hat{z}}_k^{i_o(k)}$ with $i_o(k)=\mathrm{arg} \min_{0 \leq i \leq \mathrm{it}(k) } f_k (   \mathbf{\hat{z}}^{i}_{k} )$. \rev{One advantage of this selection is that it allows us to adapt many tools used in the convergence proof of the mirror descent algorithm to the regret analysis.} Further, we choose $\mathrm{j}(k)=\mathrm{it}(k)$ in Algorithm~\ref{alg: eta changes}. 
Any other choice is in principle possible, but one has to adapt the subsequent analysis accordingly. \\
Our goal is to ensure that (\ref{accCost}) is not much larger than the total loss  $\sum_{k=1}^{T}  \,  f_k (   \mathbf{z}_{k}^c )$ incurred by any comparator sequence $  \left\lbrace  \mathbf{z}_{1}^c,   \mathbf{z}_{2}^c, \dots,   \mathbf{z}_{T}^c \right\rbrace$ satisfying $\mathbf{z}_{k}^c \in \mathcal{S}_k$. In other words, we aim to obtain a  low \textit{regret}, which we define as
\begin{align}
\label{regretNew}
R(T) & \coloneqq \sum_{k=1}^{T} \, \min_{0 \leq i \leq \mathrm{it}(k) } f_k (   \mathbf{\hat{z}}^{i}_{k} ) - \sum_{k=1}^{T}  \,  f_k (   \mathbf{z}_{k}^c ) .
\end{align}
By computing an upper bound for the regret, \rev{we can design suitable step sizes that yield a sublinear regret, i.e., the  regret bound $\mathcal{O}(\sqrt{T})$. This is a meaningful regret bound and well-known in the context of online convex  optimization since it implies that the average regret $ R(T)/T$ tends to zero for $T \rightarrow \infty$ and hence that the proposed algorithm performs well, on average as well as the comparator \cite{mokhtari2016online}. This property of the algorithm is especially desirable when the regret is used to} evaluate how well the pMHE iteration scheme performs compared to an estimation scheme that knows the optimal solutions  $  \left\lbrace  \mathbf{z}_{1}^c,   \mathbf{z}_{2}^c, \dots,   \mathbf{z}_{T}^c \right\rbrace$. Hence, we measure the real-time regret of our algorithm that carries out only finitely many optimization iterations
(due to limited hardware resources and/or minimum required sampling rate)
relative to a comparator algorithm that gets instantaneously an optimal solution from some oracle. 
 \subsection{Regret with respect to arbitrary comparator sequences}
In this section, we  establish  bounds on the regret generated by Algorithm~\ref{alg: eta changes}.  Similar to \cite{hall2013dynamical}, we derive regret bounds that  depend on the variation of the comparator sequence with respect to the dynamics $\Phi_k$ defined in (\ref{Adap: warm}):
\begin{align}
  C_{T}(\mathbf{z}_{1}^c,\cdots, \mathbf{z}_{T}^c)  &\coloneqq \sum_{k=1}^{T} \left\Vert \mathbf{z}_{k+1}^c - \Phi_k\left(\mathbf{z}_{k}^c \right) \right\Vert. \label{compSeq}
\end{align}
Moreover, we define the following notations:
\begin{align*}
G_f &\coloneqq\max_{\mathbf{z} \in \mathcal{S}_k, k > 0} \,\Vert \nabla  f_k(\mathbf{z} ) \Vert \\
 M_1 &\coloneqq   \max_{ \mathbf{z} \in  \mathcal{S}_k, k > 0} \left\Vert \nabla  \psi(\mathbf{z}) \right\Vert,  \quad \,  M_2  \coloneqq   \max_{ \mathbf{z} \in  \mathcal{S}_k, k > 0} \left\Vert \nabla  \psi(\Phi(\mathbf{z})) \right\Vert    \\ 
  M&\coloneqq   M_1 + M_2, \hspace{1.35cm}     D_{\max}  \coloneqq \max_{\mathbf{z}_1,\mathbf{z}_2 \in  \mathcal{S}_k, k > 0} D_\psi(\mathbf{z}_1,\mathbf{z}_2) ,
\end{align*}
where we assume that the maximum in each definition is well-defined. Our first main result is stated next.
\begin{theorem}
\label{theorem_iTadap}
Consider  Algorithm \ref{alg: eta changes} with $  \mathrm{j}(k)= \mathrm{it}(k) $ and any comparator sequence $  \left\lbrace  \mathbf{z}_{1}^c,   \mathbf{z}_{2}^c, \dots,   \mathbf{z}_{T}^c \right\rbrace $ with $\mathbf{z}_{k}^c \in \mathcal{S}_k$.  Let Assumptions  \ref{ass: S closed}, \ref{ass: f convex}   and \ref{ass: Bregman} hold true.  If we choose the Bregman distance $D_\psi$ such that
 \begin{align}
 \label{BregDistanceNonneg}
D_\psi(\Phi_k(\mathbf{z}), \Phi_k(\mathbf{\hat{z}})) - D_\psi(\mathbf{z},\mathbf{\hat{z}})    \leq 0 
\end{align} 
and employ  non-increasing sequences 
\begin{align}
\label{thm: non-incr sequences}
\sum_{i=0}^{\mathrm{it}(k+1)-1}\eta_{k+1}^{i} \leq \sum_{i=0}^{\mathrm{it}(k)-1}\eta_k^{i},
\end{align}
then  Algorithm \ref{alg: eta changes} gives the following regret bound
\begin{align}
\label{regretGenBound}
R(T)  
 \leq & \frac{D_{max}}{\sum_{i=0}^{\mathrm{it}(T)-1}\eta_{T}^{i}} + \frac{   G_f^2 }{2  \sigma  }\sum_{k=1}^{T}    \frac{\sum_{i=0}^{\mathrm{it}(k)-1}    (\eta_k^{i})^2  }{\sum_{i=0}^{\mathrm{it}(k)-1}\eta_k^{i}} \\ & \qquad      + \frac{M  }{\sum_{i=0}^{\mathrm{it}(T)-1}\eta_T^{i}} \sum_{k=1}^{T} \left\Vert   \mathbf{z}_{k+1}^c - \Phi_k\left(\mathbf{z}_{k}^c \right) \right\Vert \nonumber .
\end{align}
\end{theorem}
The proof of this result relies on the next lemma. 
\begin{lem}
\label{regret_betweeniT}
Consider  Algorithm \ref{alg: eta changes} with  $  \mathrm{j}(k)= \mathrm{it}(k) $  and any comparator sequence  $  \left\lbrace  \mathbf{z}_{1}^c,   \mathbf{z}_{2}^c, \dots,   \mathbf{z}_{T}^c \right\rbrace $  with $\mathbf{z}_{k}^c \in \mathcal{S}_k$.  Suppose Assumptions~\ref{ass: S closed},~\ref{ass: f convex} and~\ref{ass: Bregman} hold.  Then for a given iteration step~$i$ and a time instant $k>0$, we have that
\begin{align}
\label{result: regret_betweeniT}
  \eta_k^i &\left(f_k(\mathbf{\hat{z}}_{k}^i) - f_k( \mathbf{z}_{k}^c) \right) \\ &\leq  D_\psi( \mathbf{z}_{k}^c,\mathbf{\hat{z}}_{k}^{i}) - D_\psi(\mathbf{z}_{k}^c, \mathbf{\hat{z}}_{k}^{i+1}) +   \frac{ ( \eta_k^i)^2}{2 \sigma} \Vert \nabla f_k\left(\mathbf{\hat{z}}_{k}^{i}\right) \Vert^2. \nonumber 
\end{align}
Moreover, if  we choose the Bregman distance $D_\psi$ such that
 \begin{align}
 \label{Dcontadap}
D_\psi(\Phi_k(\mathbf{z}), \Phi_k(\mathbf{\hat{z}})) - D_\psi(\mathbf{z},\mathbf{\hat{z}})    \leq 0 ,
\end{align} 
then 
 \begin{align}
 \label{result: regret_iTasap}
    \min_{0 \leq i \leq \mathrm{it}(k) } &f_k (  \mathbf{\hat{z}}_k^{i} )  -  f_k (\mathbf{z}_{k}^c )    \\& \enskip \leq \frac{1}{\sum_{i=0}^{\mathrm{it}(k)-1}\eta_k^{i}} \Big( D_\psi \left(  \mathbf{z}_{k}^c , \mathbf{\hat{z}}_k^0 \right) - D_\psi \left( \mathbf{z}_{k+1}^c ,   \mathbf{\hat{z}}_{k+1}^0\right)  \nonumber \\&\quad   +  \frac{G_f^2}{2 \sigma}  \sum_{i=0}^{\mathrm{it}(k)-1}   (\eta_k^{i})^2  +         M \Vert \mathbf{z}_{k+1}^c - \Phi_k\left(\mathbf{z}_{k}^c  \right) \Vert   \Big) \nonumber.
  \end{align}
\end{lem}
The proof of Lemma \ref{regret_betweeniT} can be found in Appendix \ref{proof: Lemma  regret_betweeniT}. 
We are now in a position to prove the theorem.
\begin{proof}[Proof of Theorem~\ref{theorem_iTadap}]
The proof is similar to the proof of  \cite[Theorem 4]{hall2013dynamical} which derives a regret upper bound for the dynamic mirror descent in the context of online convex optimization. For ease of notation, we employ $\sum \eta_k^{i} $ to refer to the sum of all the step sizes used within the time instant $k$, i.e. to $\sum_{i=0}^{\mathrm{it}(k)-1}\eta_k^{i}$.\\
By Lemma \ref{regret_betweeniT}, (\ref{result: regret_iTasap}) holds true.
Summing  (\ref{result: regret_iTasap})  over $k~=~1,\cdots,T$ yields 
\begin{align}
\label{thm: regret_iTadapPr}
R(T) &= \sum_{k=1}^{T} \min_{0 \leq i \leq \mathrm{it}(k)} f_k (   \mathbf{\hat{z}}_k^{i} )  - \sum_{k=1}^{T}  f_k (\mathbf{z}_{k}^c   )    \\
& \leq \sum_{k=1}^{T} \frac{1}{\sum\eta_k^{i}} \Big(D_\psi \left(  \mathbf{z}_{k}^c , \mathbf{\hat{z}}_k^0 \right) - D_\psi \left( \mathbf{z}_{k+1}^c ,   \mathbf{\hat{z}}_{k+1}^0\right)  \nonumber \\&    \qquad \qquad    +  \frac{G_f^2}{2 \sigma}  \sum  (\eta_k^{i})^2  +         M  \Vert \mathbf{z}_{k+1}^c - \Phi_k\left(\mathbf{z}_{k}^c  \right) \Vert   \Big) \nonumber.
\end{align}
Using~\eqref{thm: non-incr sequences}, i.e. the fact that $\sum\eta_{k+1}^{i} \leq \sum\eta_k^{i}$, we have  
\begin{align}
\label{sumDiffBreg}
&\sum_{k=1}^{T} \frac{1}{\sum\eta_k^{i}} \, \left( D_\psi \left(  \mathbf{z}_{k}^c , \mathbf{\hat{z}}_k^0 \right) - D_\psi \left( \mathbf{z}_{k+1}^c ,   \mathbf{\hat{z}}_{k+1}^0\right) \right)  \\
&\quad= \frac{D_\psi \left(   \mathbf{z}_{1}^c  , \mathbf{\hat{z}}^0_{1} \right)}{\sum\eta_{1}^{i}}  -\frac{D_\psi \left(  \mathbf{z}_{T+1}^c ,  \mathbf{\hat{z}}^0_{T+1} \right)}{\sum\eta_{T}^{i}} \nonumber \\ &\qquad+ D_\psi \left(   \mathbf{z}_{2}^c  ,  \mathbf{\hat{z}}^0_{2} \right)    \left( \frac{1}{\sum \eta_{2}^{i}} -  \frac{1}{\sum\eta_{1}^{i}} \right) + \cdots \nonumber\\& \qquad   +D_\psi \left( \mathbf{z}_{T}^c ,  \mathbf{\hat{z}}^0_{T} \right)   \left( \frac{1}{\sum\eta_{T}^{i}} -  \frac{1}{\sum\eta_{T-1}^{i}}  \right)      \nonumber \\
& \quad\leq \frac{D_{max}}{\sum\eta_{1}^{i}}   + D_{max}  \left( \sum_{k=1}^{T-1} \frac{1}{\sum\eta_{k+1}^{i}} -  \frac{1}{\sum\eta_{k}^{i}}  \right)  =  \frac{D_{max} }{\sum\eta_{T}^{i}}  .  \nonumber 
\end{align}
Moreover,  since $\frac{1}{\sum\eta_{1}^{i}} \leq \cdots \leq \frac{1}{\sum\eta_T^{i}}$, we compute 
\begin{align}
\sum_{k=1}^{T}  \frac{M }{\sum\eta_k^{i}}  & \left\Vert  \mathbf{z}_{k+1}^c \!  - \!  \Phi_k\left(\mathbf{z}_{k}^c \right) \right\Vert  \leq \frac{M }{\sum\eta_T^{i}}     \sum_{k=1}^{T} \left\Vert  \mathbf{z}_{k+1}^c \!  - \!  \Phi_k\left(\mathbf{z}_{k}^c \right) \right\Vert  .
\end{align}
 Hence, substituting the latter  bounds into (\ref{thm: regret_iTadapPr}) yields
\begin{align}
R(T)   \leq    \frac{D_{max}}{\sum \eta_{T}^{i}} &+ \frac{   G_f^2 }{2  \sigma  }\sum_{k=1}^{T}    \frac{\sum    (\eta_k^{i})^2  }{\sum\eta_k^{i}}   \\
 &   +   \frac{ M }{\sum\eta_T^{i}} \sum_{k=1}^{T} \left\Vert   \mathbf{z}_{k+1}^c  -  \Phi_k\left(\mathbf{z}_{k}^c \right) \right\Vert, \nonumber
\end{align}
finishing the proof.
\end{proof}
Note that condition (\ref{BregDistanceNonneg}) can be  satisfied if we choose, for instance, the quadratic Bregman distance (\ref{quadBregm}) with a  weight matrix $P\in \mathbb{S}_{+\!+}^n$  that fulfills (\ref{LMIBregman}). 
We discuss in the following an important implication  of Theorem \ref{theorem_iTadap}. If we execute a single iteration per time instant, i.e.,    set $\mathrm{it}(k)=1$ for all $k> 0$, we get $ \sum_{i=0}^{\mathrm{it}(k)-1} \eta_k^i = \eta_k^0 \eqqcolon \eta_k$  in (\ref{regretGenBound}).  \rev{In this case, the condition \eqref{thm: non-incr sequences} on the step size becomes $\eta_{k+1} \leq \eta_k$ and the regret	bound \eqref{regretGenBound} is as follows} 
 \begin{align}
 \label{regretBoundit1}
R(T)  
&\leq \frac{D_{max}}{\eta_T} + \frac{   G_f^2 }{2  \sigma  }\sum_{k=1}^{T}  \eta_k    + \frac{ M}{\eta_T} \sum_{k=1}^{T} \left\Vert \mathbf{z}_{k+1}^c - \Phi_k\left(\mathbf{z}_{k}^c \right) \right\Vert.
\end{align}
\rev{This regret bound is very similar to the bound derived for the dynamic mirror descent \cite{hall2013dynamical}.} Moreover, by choosing $\eta_k= \frac{1}{\sqrt{T}}$, Algorithm~\ref{alg: eta changes} with a single optimization iteration per time instant \rev{yields 
\begin{align}
\hspace{-0.35cm}	R(T)  \! 
	&\leq\!  \sqrt{T}\left(\!D_{\max} + \frac{   G_{\mathrm{f}}^2 }{2  \sigma  }   +  M \sum_{k=1}^{T} \left\Vert \mathbf{z}_{k+1}^{\mathrm{c}} - \Phi_k\left(\mathbf{z}_{k}^{\mathrm{c}} \right) \right\Vert\! \right) 
\end{align}}
and achieves therefore a  regret bound $\mathcal{O}\big(\sqrt{T}(1+  C_{T})\big)$, \rev{where $C_T$ is defined in (\ref{compSeq}). Furthermore, if the comparator sequence is such that   $ 	C_{T}(\mathbf{z}_{1}^{\mathrm{c}},\cdots, \mathbf{z}_{T}^{\mathrm{c}})= \sum_{k=1}^{T}\left\Vert \mathbf{z}_{k+1}^{\mathrm{c}} - \Phi_k\left(\mathbf{z}_{k}^{\mathrm{c}} \right) \right\Vert =~0$, then  Algorithm \ref{alg: eta changes} achieves in this case the desired regret bound $\mathcal{O}(\sqrt{T})$ and the average regret $R(T)/T$   tends to zero when $T$ goes to infinity. }  %
In our second main result, we specify conditions under which Algorithm~\ref{alg: eta changes} attains the regret bound $\mathcal{O}\big(\sqrt{T}(1+  C_{T})\big)$  as well as GES of the estimation error.
\begin{theorem}
\label{theorem_stab_regret}
Consider  Algorithm \ref{alg: eta changes} with $  \mathrm{j}(k)= \mathrm{it}(k) $   and any comparator sequence $  \left\lbrace  \mathbf{z}_{1}^c,   \mathbf{z}_{2}^c, \dots,   \mathbf{z}_{T}^c \right\rbrace $ with $   \mathbf{z}_{k}^c  \in \mathcal{S}_k$.  Let Assumptions  \ref{ass: S closed}-\ref{ass: Bregman} hold true.  Suppose that the Bregman distance $D_\psi$ satisfies 
 \begin{align}
 \label{BregmaniTvar}
D_\psi(\Phi_k(\mathbf{z}), \Phi_k(\mathbf{\hat{z}})) - D_\psi(\mathbf{z},\mathbf{\hat{z}})    \leq  - c \, \Vert \mathbf{z} - \mathbf{\hat{z}}   \Vert^2 
\end{align} 
 and that $\mathrm{it}(k+1) \leq \mathrm{it}(k)$. Let
\begin{align}
\label{StepiTvar}
\eta_k^{i} =  \frac{\sigma }{L_f} \frac{1}{\,\sqrt{k}}, 
\end{align}
 for all $i=0,\dots, \mathrm{it}(k)-1$ and $k> 0$. Then,  the estimation error is GES and we have that
 \begin{align}
 \label{RegretiTvar}
R(T)    \leq \frac{\sqrt{T}}{\mathrm{it}(T)} \frac{ L_f }{\sigma} \Big(&  D_{max}    +    M      \sum_{k=1}^{T} \left\Vert  \mathbf{z}_{k+1}^c - \Phi_k\left(\mathbf{z}_{k}^c \right)\right\Vert \Big).
\end{align}
\end{theorem}
The proof of this result relies on the next lemma. 
\begin{lem}
\label{regret_betweeniT_eta}
Consider  Algorithm \ref{alg: eta changes} with $  \mathrm{j}(k)= \mathrm{it}(k) $   and any comparator sequence  $  \left\lbrace  \mathbf{z}_{1}^c,   \mathbf{z}_{2}^c, \dots,   \mathbf{z}_{T}^c \right\rbrace $ with $   \mathbf{z}_{k}^c  \in \mathcal{S}_k$.  Let Assumptions  \ref{ass: S closed}-\ref{ass: Bregman} hold true. If we choose the step size 
 \begin{align}
\eta_k^{i} =  \frac{\sigma }{L_f} \frac{1}{\,\sqrt{k}}, 
\end{align}
then for a given iteration step $i$ and time $k>0$, we have that 
\begin{align}
\label{result: regret_betweeniT_eta}
 \hspace{-0.3cm} \eta_k^i \left( f_k \left(\mathbf{\hat{z}}_{k}^{i+1}   \right) \! - \!  f_k(\mathbf{z}_{k}^c) \right) \leq     D_\psi \big(\mathbf{z}_{k}^c,\mathbf{\hat{z}}_{k}^i\big) \!    -    \!    D_\psi \big(\mathbf{z}_{k}^c, \mathbf{\hat{z}}_{k}^{i+1} \big)   . 
\end{align}
Moreover, if we choose the Bregman distance $D_\psi$ such that
 \begin{align}
 \label{Dcontadap_eta}
D_\psi(\Phi_k(\mathbf{z}), \Phi_k(\mathbf{\hat{z}})) - D_\psi(\mathbf{z},\mathbf{\hat{z}})    \leq 0, 
\end{align} 
then 
 \begin{align}
 \label{result: regret_iTasap_eta}
&   \min_{0 \leq i \leq \mathrm{it}(k)} f_k (  \mathbf{\hat{z}}_k^{i} )  -  f_k (\mathbf{z}_{k}^c )    \\& \quad\leq \frac{1}{\sum_{i=0}^{\mathrm{it}(k)-1}\eta_k^{i}} \Big( D_\psi \left(  \mathbf{z}_{k}^c , \mathbf{\hat{z}}_k^0 \right) - D_\psi \left( \mathbf{z}_{k+1}^c ,   \mathbf{\hat{z}}_{k+1}^0\right)  \nonumber \\&  \hspace{3cm}+      \!  M \Vert \mathbf{z}_{k+1}^c - \Phi_k\left(\mathbf{z}_{k}^c  \right) \Vert   \Big) \nonumber.
  \end{align}
\end{lem}
The proof of Lemma \ref{regret_betweeniT_eta} can be found in Appendix \ref{proof: Lemma  regret_betweeniT_eta}. 
We are now in a position to prove  Theorem~\ref{theorem_stab_regret}.
\begin{proof}[Proof of Theorem~\ref{theorem_stab_regret}]
GES of the estimation error follows, since  $\eta_k^{i}$ in (\ref{StepiTvar}) satisfies  (\ref{stepSizeadap}), i.e.  $\eta_k^{i} \leq \frac{\sigma}{L_f}$.  
Note that, by Lemma \ref{regret_betweeniT_eta}, i.e. (\ref{result: regret_iTasap_eta}),
\begin{align}
R(T) &= \sum_{k=1}^{T} \min_{0 \leq i \leq \mathrm{it}(k) } f_k (   \mathbf{\hat{z}}_k^{i} )  - \sum_{k=1}^{T}  f_k (\mathbf{z}_{k}^c   )    \\
& \leq \sum_{k=1}^{T} \frac{1}{\sum_{i=0}^{\mathrm{it}(k)-1}\eta_k^{i}} \Big(D_\psi \left(  \mathbf{z}_{k}^c , \mathbf{\hat{z}}_k^0 \right) - D_\psi \left( \mathbf{z}_{k+1}^c ,   \mathbf{\hat{z}}_{k+1}^0\right)  \nonumber \\&  \hspace{3.2cm} +        M \Vert \mathbf{z}_{k+1}^c - \Phi_k\left(\mathbf{z}_{k}^c  \right) \Vert   \Big) \nonumber.
\end{align}
Since $\mathrm{it}(k+1) \leq \mathrm{it}(k)$,  we have that
\begin{align}
\label{StepSizeIneq}
\sum_{i=0}^{\mathrm{it}(k)-1}\eta_k^{i}   =  \frac{\sigma}{L_f} \frac{\mathrm{it}(k)}{\sqrt{k}} & \geq  \frac{\sigma}{L_f} \frac{\mathrm{it}(k+1)}{\sqrt{k+1}} = \sum_{i=0}^{\mathrm{it}(k+1)-1}\eta_{k+1}^{i}.
\end{align}
Hence,~\eqref{thm: non-incr sequences} holds and as a consequence, we can derive an upper bound, similar to~\eqref{sumDiffBreg}, to obtain
\begin{align}
\label{comb1}
&\sum_{k=1}^{T} \frac{1}{\sum_{i=0}^{\mathrm{it}(k)-1}\eta_k^{i}} \, \left( D_\psi \left(  \mathbf{z}_{k}^c , \mathbf{\hat{z}}_k^0 \right) - D_\psi \left( \mathbf{z}_{k+1}^c ,   \mathbf{\hat{z}}_{k+1}^0\right) \right)  \\& \leq  \frac{D_{max} }{\sum_{i=0}^{\mathrm{it}(T)-1}\eta_{T}^{i}} =   \frac{D_{max} L_f \sqrt{T}}{ \mathrm{it}(T) \sigma} .  \nonumber 
\end{align}
Moreover, given (\ref{StepSizeIneq}), we have that 
\begin{align}
\label{comb2}
\sum_{k=1}^{T} & \frac{M}{\sum_{i=0}^{\mathrm{it}(k)-1}\eta_k^{i}}    \left\Vert  \mathbf{z}_{k+1}^c - \Phi_k\left(\mathbf{z}_{k}^c \right) \right\Vert \\ & \leq \frac{M }{\sum_{i=0}^{\mathrm{it}(T)-1}\eta_T^{i}}     \sum_{k=1}^{T} \left\Vert  \mathbf{z}_{k+1}^c - \Phi_k\left(\mathbf{z}_{k}^c \right) \right\Vert \nonumber
\\ & =\frac{ML_f \sqrt{T} }{  \mathrm{it}(T) \sigma}     \sum_{k=1}^{T} \left\Vert  \mathbf{z}_{k+1}^c - \Phi_k\left(\mathbf{z}_{k}^c \right) \right\Vert \nonumber.
\end{align}
Combining (\ref{comb1}) and (\ref{comb2})  completes the proof.
\end{proof}
A direct consequence of Theorem \ref{theorem_stab_regret} is that fixing  the number of optimization iterations $\mathrm{it}(k)=\mathrm{it}(k+1) \eqqcolon \mathrm{it}$ \rev{and increasing $\mathrm{it}$ lead to a smaller regret bound. This allows for a trade-off between computational effort and performance}. In fact, if the comparator sequence $  \left\lbrace  \mathbf{z}_{1}^c,   \mathbf{z}_{2}^c, \dots,   \mathbf{z}_{T}^c \right\rbrace$ follows the dynamics described by the a priori estimate operator $\Phi_k$ closely, and if we let $ \mathrm{it} \rightarrow \infty$, then the bound in Theorem \ref{theorem_stab_regret} vanishes and we obtain an algorithm with zero regret, \rev{i.e., $\lim\limits_{\mathrm{it} \rightarrow \infty} R(T)  \rightarrow 0$.} 
\\
We also remark that the condition $\mathrm{it}(k+1) \leq \mathrm{it}(k)$ requires that we employ a smaller or equal number of optimization iterations each time we receive a new measurement. This condition is in line with the intuitive observation that it is  preferable to execute more  iterations at the beginning of the pMHE iteration scheme,
since our regret measure is aggregated over time and thus memorizes initially  poor estimates. 

\begin{table*}[!t]
	\caption{Summary of  results.  We employ the following notation for abbreviation:  $\Delta_{\Phi_k} D_\psi (\mathbf{z},\mathbf{\hat{z}})   \coloneqq D_\psi(\Phi_k(\mathbf{z}), \Phi_k(\mathbf{\hat{z}})) - D_\psi(\mathbf{z},\mathbf{\hat{z}})$ and $ 
		\sum_{i}\eta_k^{i}  \coloneqq \sum_{i=0}^{\mathrm{it}(k)-1}\eta_k^{i}$.}
	\label{Table: results}
	\begin{center}
		{\setlength{\extrarowheight}{4pt}
			\begin{tabular}{ l  l  l  l} 
				\hline 
				Theorem & Assumptions  & Step size &  Result \\ [1ex] 
				\hline  
				Thm. \ref{prop: Stabadap} &   \begin{tabular}{@{}l@{}}   A1 - A4       \\ $  \Delta_{\Phi_k} D_\psi (\mathbf{z},\mathbf{\hat{z}}) \leq  - c \, \Vert \mathbf{z} - \mathbf{\hat{z}}   \Vert^2   $\end{tabular}    & $ \eta_k^i \leq \frac{ \sigma }{L_f}$  &   Stability \\[3ex] \hline  
				Thm. \ref{theorem_iTadap}    &    \begin{tabular}{@{}l@{}}   A1, A2, A4       \\ $  \Delta_{\Phi_k} D_\psi (\mathbf{z},\mathbf{\hat{z}}) \leq 0  $\end{tabular}  &  $ \sum\limits_{i}\eta_{k+1}^{i}   \leq \sum\limits_{i}\eta_k^{i}   $       & Regret: {$\!\begin{aligned}   R(T) 
						\leq & \frac{D_{max}}{\sum_{i}\eta_{T}^{i}} + \frac{   G_f^2 }{2  \sigma  }\sum_{k=1}^{T}    \frac{\sum_{i}   (\eta_k^{i})^2  }{\sum_{i}\eta_k^{i}}      + \frac{M }{\sum_{i}\eta_T^{i}} C_T  
					\end{aligned}$}  \\[3ex]    \hline 
				Thm. \ref{theorem_stab_regret}  &    \begin{tabular}{@{}l@{}}   A1 - A4       \\  $  \Delta_{\Phi_k} D_\psi (\mathbf{z},\mathbf{\hat{z}}) \leq  - c \, \Vert \mathbf{z} - \mathbf{\hat{z}}   \Vert^2   $ \\ $      \mathrm{it}(k+1) \leq \mathrm{it}(k) $ \end{tabular} & $\eta_k^i= \frac{\sigma}{L_f} \frac{1}{\sqrt{k}}  $ &  Stability + Regret: {$\!\begin{aligned} R(T)  
						&\leq \frac{\sqrt{T}}{ \mathrm{it}(T)}  \frac{ L_f}{\sigma } \Big(   D_{max}  +  M     C_T  \Big)     
					\end{aligned}$}   \\[4ex]  \hline 
				Thm. \ref{thm: constantReg}  &  \begin{tabular}{@{}l@{}}   A1 - A5     \\ $  \Delta_{\Phi_k} D_\psi (\mathbf{z},\mathbf{\hat{z}}) \leq  - c \, \Vert \mathbf{z} - \mathbf{\hat{z}}   \Vert^2   $\end{tabular}   &    ${  \eta_k^i }\leq { {\frac{\sigma }{L_f}}}  $  &  \begin{tabular}{@{}l@{}}   Stability + Regret: \\   {$\!\begin{aligned}   
							R(T)   \leq  \frac{L_f}{2}\,  \frac{\alpha^2\,  \beta^2 }{1 -\beta^2  } \Vert   \mathbf{z}_{0}-\mathbf{\bar{z}}_0 \Vert^2   +  \frac{L_f}{2}\,  \frac{ \alpha_c^2 \, \beta_c^2}{1 -\beta_c^2  } \Vert   \mathbf{z}_{0}-\mathbf{z}^c_0 \Vert^2 
						\end{aligned}$} \end{tabular}    \\ [3ex] 
				\hline
		\end{tabular}}
	\end{center}
\end{table*}
 \subsection{Regret with respect to exponentially stable comparator sequences}
 As we mentioned before, in general, there is no requirement that  the comparator sequence converges to the true state. This being said, it is reasonable to restrict the class of comparator sequences to sequences that converge exponentially fast to the true state. 
    We study this case in this subsection by imposing  the following  additional assumption. 
    \begin{ass}[Exponentially stable comparator sequence]  \label{convComp}   The comparator sequence $  \left\lbrace  \mathbf{z}_{1}^c,   \mathbf{z}_{2}^c, \dots,   \mathbf{z}_{T}^c \right\rbrace$ with initial guess $\mathbf{z}_{0}^c$ is generated from a state estimator that yields GES error dynamics. More specifically, there exists   positive constants $\alpha_c \geq 1$ and $0\leq \beta_c <1$ such that 
\begin{align}
 \Vert \mathbf{z}_{k}  - \mathbf{z}_{k}^c \Vert \leq \alpha_c \enskip  \beta_c^{k} \enskip \Vert  \mathbf{z}_{0}  - \mathbf{z}_{0}^c \Vert 
\end{align}
holds for each $0 <  k \leq T$.  Here, $\mathbf{z}_{k}^c = \begin{bmatrix}  x_{k-N}^c \\[0.1cm] \mathbf{\hat{w}}_{k}^c   \end{bmatrix}$.
    \end{ass}
    Notably, when the comparator sequence satisfies the exponential stability assumption, Algorithm~\ref{alg: eta changes} leads to constant  regret, as our next result shows.  
 \begin{theorem}
 \label{thm: constantReg}
Consider  Algorithm \ref{alg: eta changes}  and let Assumptions  \ref{ass: S closed}-\ref{ass: Bregman} hold true. Suppose that a comparator sequence $  \left\lbrace  \mathbf{z}_{1}^c,   \mathbf{z}_{2}^c, \dots,   \mathbf{z}_{T}^c \right\rbrace  $  is generated from a GES  estimator with  
initial guess $\mathbf{z}_{0}^c$, as in Assumption \ref{convComp}.
 If the Bregman distance $D_\psi$ satisfies 
 \begin{align}
 \label{BreAss}
D_\psi(\Phi_k(\mathbf{z}), \Phi_k(\mathbf{\hat{z}})) - D_\psi(\mathbf{z},\mathbf{\hat{z}})  \leq  - c \Vert \mathbf{z} - \mathbf{\hat{z}}   \Vert^2,   
\end{align} 
for all $ \mathbf{z},\mathbf{\hat{z}} \in \mathbb{R}^{(N+1)n}$ and  $\eta_k^i \leq \frac{ \sigma }{L_f} $, then the estimation error is GES and  
\begin{align}
R(T)  & \leq  \frac{L_f}{2}\,   \frac{\alpha^2\,  \beta^2 }{1 -\beta^2  } \Vert   \mathbf{z}_{0}-\mathbf{\bar{z}}_0 \Vert^2  +      \frac{L_f}{2}\,  \frac{ \alpha_c^2 \, \beta_c^2}{1 -\beta_c^2  } \Vert   \mathbf{z}_{0}-\mathbf{z}^c_0 \Vert^2 , 
\end{align}
with  $ \beta \coloneqq \sqrt{ 1 - \frac{2c}{\gamma}} \in [0,1) $ and $\alpha \coloneqq {\sqrt{ \gamma / \sigma}}$.
\end{theorem}
\begin{proof}
In view of Theorem \ref{prop: Stabadap}, GES holds since the Lyapunov function $
 V\big(  \mathbf{z}_{k}, \mathbf{\hat{z}}^{\mathrm{j}(k)}_{k}\big)  = D_\psi\big( \mathbf{z}_{k},\mathbf{\hat{z}}^{\mathrm{j}(k)}_{k}\big)$
 satisfies 
\begin{align}
\Delta V &= V\big( \mathbf{z}_{k},\mathbf{\hat{z}}^{ \mathrm{j}(k)}_{k}\big) - V\big(  \mathbf{z}_{k-1},\mathbf{\hat{z}}^{ \mathrm{j}(k-1)}_{k-1}\big)   \\& \leq -c \big\Vert  \mathbf{z}_{k-1}-\mathbf{\hat{z}}^{\mathrm{j}(k-1)}_{k-1} \big\Vert^2  \nonumber.
\end{align}
In particular, this implies based on (\ref{toreplace})  that 
\begin{align}
 \label{decreaseIT}
 D_\psi\big(    \mathbf{z}_{k},\mathbf{\hat{z}}^{ \mathrm{j}(k)}_{k}  \big) \leq \beta_e^{k} \, D_\psi\big( \mathbf{z}_{0},\mathbf{\bar{z}}_0\big),
\end{align}
where $\beta_e \coloneqq 1 - \frac{2c}{\gamma} \in [0,1)$.  Given that $D_\psi$ is strongly convex and strongly smooth, we have that 
\begin{align} 
\big\Vert  \mathbf{z}_{k} -\mathbf{\hat{z}}^{\mathrm{j}(k )}_{k}     \big\Vert^2
& \leq \frac{2 }{\sigma}\, \beta_e^{k} \, D_\psi\big(    \mathbf{z}_{0},\mathbf{\bar{z}}_0  \big) \leq \frac{\gamma}{\sigma}\, \beta_e^{k} \,  \big\Vert   \mathbf{z}_{0}-\mathbf{\bar{z}}_0   \big\Vert^2. 
\end{align}
With $\beta \coloneqq \sqrt{\beta_e}=\sqrt{1-2c/\gamma} \in [0,1)$ and  $\alpha \coloneqq{\sqrt{ \gamma / \sigma}} \geq 1$, we obtain that
\begin{align} 
\label{lemTh4result}
\big\Vert  \mathbf{z}_{k} -\mathbf{\hat{z}}^{\mathrm{j}(k )}_{k}     \big\Vert
& \leq \alpha \,\beta^{k} \, \Vert   \mathbf{z}_{0}-\mathbf{\bar{z}}_0 \Vert.
\end{align}
The regret can be upper bounded as follows
\begin{align}
R(T) &= \sum_{k=1}^{T} \quad \min_{0 \leq i \leq \mathrm{it}(k) } f_k (  \mathbf{\hat{z}}^{i}_{k} ) - \sum_{k=1}^{T}  \,  f_k (   \mathbf{z}_{k}^c ) \\
 &\leq  \sum_{k=1}^{T} f_k  \big( \mathbf{\hat{z}}^{\mathrm{j}(k )}_{k}   \big)  - \sum_{k=1}^{T}  \,  f_k \big( \mathbf{z}_{k}^c \big). \nonumber
\end{align}
Furthermore, 
\begin{align}
\label{fDiff}
    f_k  \big( \mathbf{\hat{z}}^{\mathrm{j}(k )}_{k}   \big) \! - \!    f_k \big( \mathbf{z}_{k}^c \big) &=  f_k  \big( \mathbf{\hat{z}}^{\mathrm{j}(k )}_{k}   \big)\! - \! f_k  \big( \mathbf{z}_{k}   \big)  +  f_k  \big( \mathbf{z}_{k}   \big) \! -  \!  f_k \big( \mathbf{z}_{k}^c \big)\nonumber \\ 
     & \leq  \big\vert  f_k  \big( \mathbf{\hat{z}}^{\mathrm{j}(k )}_{k}   \big) - f_k  \big( \mathbf{z}_{k}   \big)  \big\vert \nonumber  \\& \hspace{1.5cm}  +  \big\vert     f_k \big( \mathbf{z}_{k}^c \big) -  f_k  \big( \mathbf{z}_{k}   \big) \big\vert. 
\end{align}
By Assumption  \ref{ass: f strngSmooth}, we have that for any  $\mathbf{z} \in \mathbb{R}^{(N+1)n}$ 
\begin{align}
\label{lipTrue}
 f_k ( \mathbf{z}   )  \leq  f_k ( \mathbf{z}_{k}   ) &+\nabla f_k  ( \mathbf{z}_{k}    )^\top   ( \mathbf{z} -\mathbf{z}_{k} )     + \frac{L_f}{2}  \Vert  \mathbf{z}_{k} - \mathbf{z}  \Vert^2.  
\end{align}
Since $f_k$ achieves its minimal value at $ \mathbf{z}_{k} $ by Assumption \ref{ass: f convex}, $\nabla f_k \big( \mathbf{z}_{k}   \big) = 0$ and we obtain in (\ref{lipTrue}) for $\mathbf{z}=\mathbf{\hat{z}}^{\mathrm{j}(k )}_{k}  $ 
\begin{align}
0 \leq f_k \big( \mathbf{\hat{z}}^{\mathrm{j}(k )}_{k}   \big) -  f_k \big( \mathbf{z}_{k}   \big)  \leq \frac{L_f}{2} \big\Vert  \mathbf{z}_{k} -  \mathbf{\hat{z}}^{\mathrm{j}(k )}_{k}   \big\Vert ^2.
\end{align}
Similarly, we have for $\mathbf{z}= \mathbf{z}_{k}^c$    in (\ref{lipTrue})
\begin{align}
0 \leq f_k \big(\mathbf{z}_{k}^c   \big) -  f_k \big( \mathbf{z}_{k}   \big)  \leq \frac{L_f}{2} \big\Vert  \mathbf{z}_{k} - \mathbf{z}_{k}^c   \big\Vert ^2. 
\end{align}
Substituting the latter two inequalities into (\ref{fDiff}) yields
\begin{align}
 f_k  \big( \mathbf{\hat{z}}^{\mathrm{j}(k )}_{k}   \big)  -     f_k \big( \mathbf{z}_{k}^c \big) \leq  \frac{L_f}{2} \big\Vert \mathbf{z}_{k} -  \mathbf{\hat{z}}^{\mathrm{j}(k )}_{k}   \big\Vert ^2 +\frac{L_f}{2} \big\Vert  \mathbf{z}_{k} - \mathbf{z}_{k}^c   \big\Vert ^2. 
\end{align}
By Assumption \ref{convComp} and (\ref{lemTh4result}), we obtain
\begin{align} 
f_k \big( \mathbf{\hat{z}}^{\mathrm{j}(k )}_{k} \big) \! - \! f_k  \big( \mathbf{z}_{k}^c  \big) &    \leq \frac{L_f}{2} \,  \alpha^2 \,\beta^{2k} \, \Vert   \mathbf{z}_{0}-\mathbf{\bar{z}}_0 \Vert^2 \\& \hspace{1cm} +  \frac{L_f}{2} \,  \alpha_c^2 \enskip  \beta_c^{2k} \enskip \Vert  \mathbf{z}_{0}  - \mathbf{z}_{0}^c \Vert^2  \nonumber.
\end{align}
Hence,
\begin{align}
\label{regretKappa}
\hspace{-0.2cm} R(T)  & \!\leq \sum_{k=1}^{T} \!   \frac{L_f}{2}\, \alpha^2 \,\beta^{2k} \, \Vert   \mathbf{z}_{0}-\mathbf{\bar{z}}_0 \Vert^2  \! + \frac{L_f}{2}  \, \alpha_c^2 \enskip  \beta_c^{2k} \enskip \Vert  \mathbf{z}_{0}  - \mathbf{z}_{0}^c \Vert^2  . 
\end{align}
Since  $ \beta \in [0,1)$, $ \beta^2 \in [0,1)$   and we have
\begin{align}
\sum_{k=1}^{T}  (\beta^2)^k = \frac{\beta^2- \beta^{2(T+1)}  }{1 -\beta^2  } \leq \frac{\beta^2 }{1 -\beta^2  }.
\end{align}
Therefore, it holds that
\begin{align}
  \sum_{k=1}^{T}  \frac{L_f}{2}\, \alpha^2 \,\beta^{2k} \, \Vert   \mathbf{z}_{0}-\mathbf{\bar{z}}_0 \Vert^2  &\leq   \frac{L_f}{2}  \Vert   \mathbf{z}_{0}-\mathbf{\bar{z}}_0 \Vert^2 \frac{ \alpha^2\,\beta^2 }{1 -\beta^2 }.
\end{align}
By carrying out a similar analysis for the second sum in (\ref{regretKappa}), the desired regret upper bound can be obtained.
\end{proof}
We summarize the obtained results of the paper in Table~\ref{Table: results}.
 

%% file: example.tex
\section{Simulation results}
\label{sec: sim}
In order to demonstrate the stability and performance properties of the anytime pMHE algorithm, we consider \rev{the following} discrete-time linear system of the form (\ref{system}), where 
\begin{align}
	\label{matricesDis}
	\begin{split}
		A&=\begin{bmatrix}
			0.8831& 0.0078 &0.0022 \\
			0.1150 &0.9563 &0.0028\\
			0.1178& 0.0102 &0.9954
		\end{bmatrix},  \qquad 
		B=\begin{bmatrix}
			0 \\0\\0
		\end{bmatrix}, \\
		C&= \begin{bmatrix} 32.84 &32.84 &32.84 \end{bmatrix} 
	\end{split}
\end{align}
with $(A,C)$ is detectable. \rev{This system is taken from  \cite{sui2014linear}, where the nonlinear model of a well-mixed, constant volume, isothermal batch reactor is  linearized and discretized with a sampling time of $T_s = 0.25$. The associated (continuous-time) nonlinear system can be found in \cite[Section 3]{haseltine2005critical}.}    Given that the states represent concentrations, they are constrained to be nonnegative, i.e., $x_k \geq 0$.  We employ the proposed anytime pMHE scheme introduced in Algorithm \ref{alg: eta changes} with the horizon length of $N=~2$ and designed such that the assumptions and conditions of Theorem~\ref{theorem_stab_regret} are fulfilled. For the a priori estimate, we choose $\mathrm{j}(k)=\mathrm{it}(k)$ and design the observer gain $L$ in (\ref{Adap: warm}) such that the eigenvalues of $A-LC$ are given by $\lambda=\begin{bmatrix}
0.4754 & 0.8497 & 0.9727 \end{bmatrix}$. Moreover, we only consider the first state in the horizon window $\hat{x}_{k-N}$ as decision variable, i.e., we set the stage cost $q$ in (\ref{N>1Unc})  and the model residual $\hat{w}_i$ to be zero. The stage cost $r$ in  (\ref{N>1Unc}) is chosen as $r(x) = \frac{1}{2} \Vert x \Vert_R^2$ with $R=0.01$. The resulting sum of stage costs at time $k$ is  
\begin{align}
\label{examSumStag}
f_k(x) =\frac{1}{2}   \sum_{i=k-N}^{k-1} \left\Vert y_i -  C A^{i-k+N}\, x  \right\Vert_R^2.
\end{align}
Furthermore, we choose  the quadratic Bregman distance $D_\psi (x_1 ,x_2)=\frac{1}{2} \Vert x_1 -x_2 \Vert_P^2$. To satisfy the stability condition~\eqref{BregmaniTvar}, we design the weight matrix $P \succeq 0 $ such that the LMI \eqref{LMIBregman} is satisfied. In addition, we fix the number of iterations $\mathrm{it}(k)$, i.e., $\mathrm{it}(k)=\mathrm{it}(k+1) \eqqcolon \mathrm{it}$.
 The step sizes are chosen as (\ref{StepiTvar}), i.e., $\eta_k^{i} =  \frac{\sigma }{L_f} \frac{1}{\sqrt{k}} $. Here, $\sigma$ denotes the strong convexity parameter of the Bregman distance which is given by $\sigma= \min(\lambda_i(P))$.  The constant $L_f$ is the strong smoothness parameter of $f_k$ defined in (\ref{examSumStag}). It can be computed as 
 \begin{align}
     L_f=  R \sum_{i=k-N}^{k-1}  \left\Vert C A^{i-k+N} \right\Vert^2.
 \end{align}
As our estimate, we select at each time $k$ the iterate $\mathbf{\hat{z}}_k^{i_o(k)}$ with the minimal cost, i.e., $i_o(k)=\mathrm{arg} \min_{0 \leq i \leq \mathrm{it}(k) } f_k (   \mathbf{\hat{z}}^{i}_{k} )$. We compare the obtained stability results with the Luenberger observer designed with the same matrix $L$, as well as with those obtained from Algorithm \ref{alg: optimal pMHE}, where the pMHE scheme is based on solving (\ref{condensedpMHE}). For this estimator, we choose the same design parameters of the anytime pMHE iteration scheme given by $N$, $f_k$, $D_\psi$ and $L$. The resulting estimation errors for each estimation strategy are shown in Figure \ref{Fig:ErrorCompEstm}.  
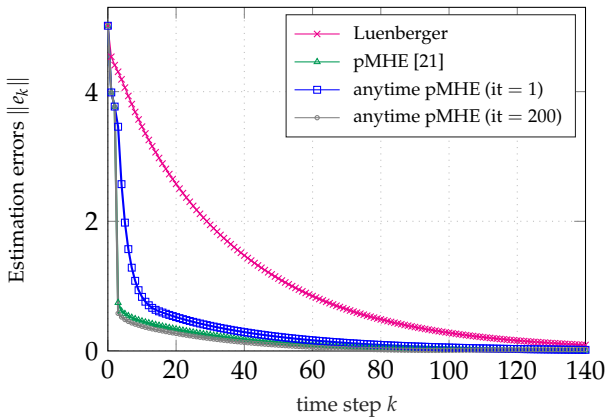
\begin{figure}[ht]
\centering
\input{figures/errors_diff_est.tikz}
  \caption{The evolution of the estimation
errors corresponding to the employed estimation strategies over time. }
\label{Fig:ErrorCompEstm}
\end{figure}%
All estimators exhibit GES of the estimation errors. This includes the case where we execute only one iteration of the optimization algorithm  per time instant $k$, i.e., $\mathrm{it}=1$. Note that for a small number of iterations, the choice of the observer gain $L$ affects the performance of the estimator. In this case, it is useful to tune $L$ such that a satisfactory performance is attained. Nevertheless, if we perform $\mathrm{it}=200$ iterations for example, the choice of $L$  does not have much impact on performance and we can observe that the iteration scheme performs even better than  Algorithm \ref{alg: optimal pMHE}. We illustrate  the effect of increasing the number of iterations on the convergence of the estimation error in Figure \ref{Fig:ErrorComp}.    We can see that the more we iterate, the faster is the convergence of the estimation error to zero. 
\begin{figure}[ht]
  \centering
\input{figures/errorsiT.tikz}
\caption{The evolution of the estimation
errors corresponding to  anytime pMHE with different number of iterations  over time.}
\label{Fig:ErrorComp}
\end{figure}
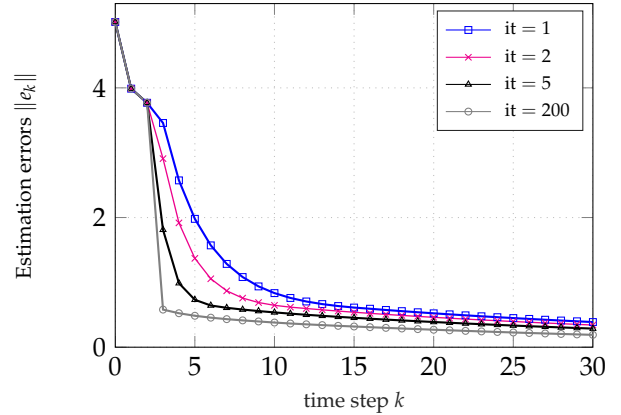
~\\
\rev{We also compare the proposed pMHE iteration scheme with the MHE approach in \cite{alessandri2017fast}, in which single and multiple iterations of descent methods are performed  each time a new measurement becomes available. Given the similarity between the underlying optimization algorithms, we employ the gradient descent for the MHE algorithm in \cite{alessandri2017fast}, which is refered to as GMHE, with $\mathrm{it}\in \mathbb{N}_+$ iterations at each time instant. In the associated cost function, we select the same sum of stage cost (\ref{examSumStag}) as in anytime pMHE. In GMHE, the a priori estimate is set to $\bar{x}_{k-N}=A\, \hat{x}_{k-N-1}^{\mathrm{it}}$. Moreover, GES of the estimation error can be ensured through a suitable condition on the step size used in the iteration step of the gradient descent \cite[Corollary 1]{alessandri2017fast}.  However, for this example, and after performing many numerical tests, we were not able to find a suitable value of the step size that satisfies this condition. Nevertheless, we tested the approach for arbitrary values of the step size and observed convergence of the estimation error to zero.  For these values, we computed the resulting root mean square error (RMSE)
	\begin{align}
		\mathrm{RMSE}=\sqrt{\sum_{k=N}^{T_{\text{sim}}} \frac{\Vert e_k \Vert^2}{T_\text{sim} - N+1}},
	\end{align}
	where $T_{\text{sim}}=100$  denotes the simulation time. For example, if we perform a single iteration per time instant and if the step size is chosen as the pMHE step size, we obtain $1.2694$ for GMHE and $1.0913$ for anytime pMHE. Note that  if we additionally construct the a priori estimate in GMHE based on the Luenberger observer (as is the case in pMHE), we obtain the exact same state estimates. If the step size in GMHE is chosen as $0.01$, the RMSE generated by GMHE becomes smaller than that of anytime pMHE and has the value $0.9602$. However, if we perform $\mathrm{it}=10$ iterations per time instant, pMHE performs better than GMHE ($0.9930$ vs  $ 1.0094$). This again demonstrates that the bias of the Luenberger observer in anytime pMHE is eventually fading away with each iteration. In GMHE, however, increasing the number of iterations to $\mathrm{it}=10$ does not seem to yield an improved performance (in fact it is worse than that with $\mathrm{it}=1$). Summarizing, although the optimization algorithm in both GMHE and anytime pMHE consists of gradient descent steps, a suitable step size in GMHE that ensures exponential stability of the estimation error is not always easy to compute, which is also remarked by the authors in \cite{alessandri2017fast}. Nevertheless, we observed in simulations that employing  the Newton method instead yields much better results for the approach in \cite{alessandri2017fast}. This is due to the fact that the  cost function is quadratic, which implies that the MHE problem is solved after one iteration. Moreover, similar to anytime pMHE, the underlying sufficient condition for stability can easily be fulfilled. Since we do not cover the use of the Newton method in the pMHE algorithm, we omit the carried out comparisons due to space
	constraints.} \\In the following, we investigate for anytime pMHE the regret \eqref{regretNew} with respect to the comparator sequence  given by the true states $x_{k-N}$. Note that  
\begin{align}
f_k(x_{k-N}) =\frac{1}{2}   \sum_{i=k-N}^{k-1} \left\Vert y_i -  C A^{i-k+N}\, x_{k-N}  \right\Vert_R^2 = 0.
\end{align}
We  employ different  number of iterations  for each  pMHE iteration scheme. After each simulation time $T$, we compute and plot the resulting regrets $R(T)$ as well as the average regrets $R(T) /T $  in Figure \ref{Fig:Regret Diff iT} and \ref{Fig:AvrRegret Diff iT}, respectively. Moreover, we plot the regret associated to Algorithm \ref{alg: optimal pMHE} when compared with this optimal sequence of true states.  
\begin{figure}[ht]
  \centering
\input{figures/regret_non_aver_Diff_iT_N2.tikz}
\caption{The resulting regrets of anytime pMHE schemes with different number of optimization  iterations.}
\label{Fig:Regret Diff iT}
\end{figure}
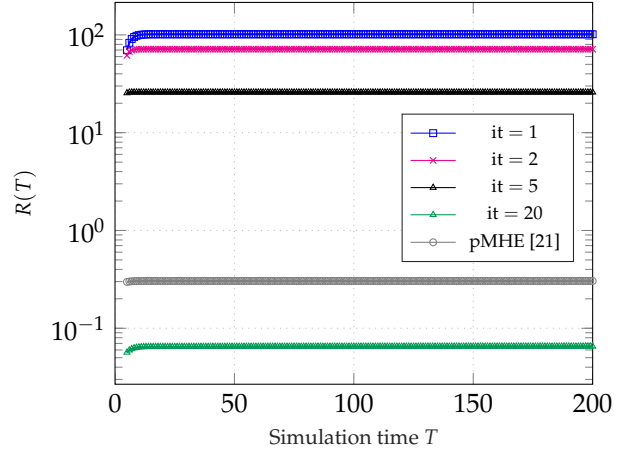
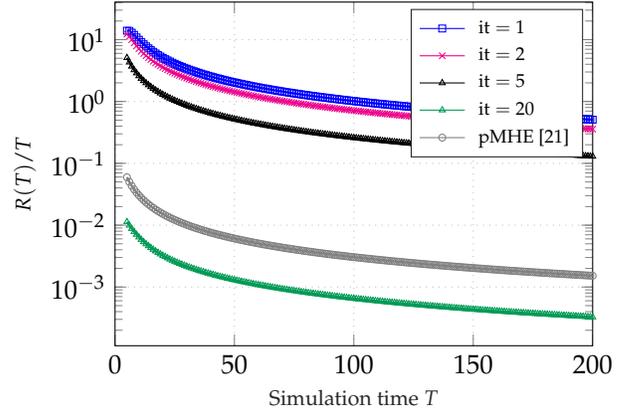
\begin{figure}[ht]
  \centering
\input{figures/regret_Diff_iT_N2.tikz}
\caption{The resulting average regret of anytime pMHE schemes  with different number of optimization  iterations.}
\label{Fig:AvrRegret Diff iT}
\end{figure}%
~\\\rev{We can see that anytime pMHE exhibits a sublinear regret  in Figure \ref{Fig:Regret Diff iT}  and that the average regret $R(T)/T$  tends to zero for $T\rightarrow \infty$  in Figure \ref{Fig:AvrRegret Diff iT}. Note that one could also deduce the qualitative behavior of the average regret directly from Figure~\ref{Fig:Regret Diff iT}, since a sublinear regret charachterized by the regret bound  $\mathcal{O}(\sqrt{T})$ implies for the average regret that $\lim\limits_{T \rightarrow \infty} R(T)/T \rightarrow 0$.} Observe  also that we can achieve lower regrets by increasing the number of iterations. This observation is in line with the regret upper bound (\ref{RegretiTvar}) obtained in Theorem~\ref{theorem_stab_regret}. Moreover, we can see that the regret of the iteration scheme with $\mathrm{it}=20$ is  lower than the regret of Algorithm~\ref{alg: optimal pMHE}, \rev{in which the solution of the optimization problem \eqref{condensedpMHE} is computed at each time instant. This is due to the novel warm-start strategy in the proposed approach; although stability of the pMHE algorithm is induced from the Luenberger observer, it is only used in the a priori estimate to warm-start the optimization algorithm. Hence, its bias is fading away each time we perform the optimization iteration step (\ref{optAdap}) and an improved performance can be achieved with each iteration.  In Algorithm \ref{alg: optimal pMHE}, however, we can see that the solution of \eqref{condensedpMHE} is designed to lie in proximity of the a priori estimate. This implies that the suboptimal bias of the Luenberger observer is present in each internal iteration of the optimization algorithm used to solve the underlying pMHE problem, which indicates that increasing the number of iterations in this case might not yield to a smaller regret.}   In fact, in order to \rev{validate this observation via simulations} and illustrate the impact of Luenbrger observer, we also compute the  regret  $R(T)$ of the pMHE scheme in which, instead of centering the Bregman distance  around the previous iterate (see \eqref{optAdap}), we use  
      \begin{align}
  \label{UpdateWarmCons}
  \mathbf{\hat{z}}^{i+1}_{k}=\argmin_{\quad\mathbf{z}\in \mathcal{S}_k}\left\lbrace  \eta_k^i\,    \nabla f_k\left(  \mathbf{\hat{z}}^{i}_{k} \right)^\top \mathbf{z}+ D_\psi (\mathbf{z},  \mathbf{\bar{z}}_{k}) \right\rbrace .  
   \end{align}
In this case, the Bregman distance is always centered around the current a priori estimate $\mathbf{\bar{z}}_{k}$ given by the Luenbeger observer. The results are depicted in Figure \ref{Fig:Regret Diff iT WarmConst}.   
\begin{figure}[ht]
  \centering
\input{figures/regret_non_aver_Diff_iT_N2_warmConst.tikz}
\caption{The resulting regret for the pMHE scheme with update step (\ref{UpdateWarmCons}) and $\mathrm{it}$ iterations at each time instant $k$. }
\label{Fig:Regret Diff iT WarmConst}
\end{figure}
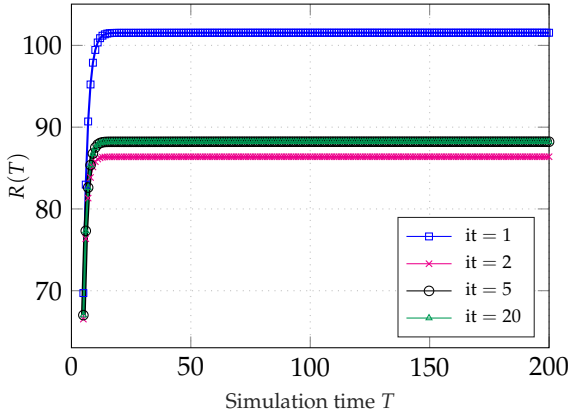
As demonstrated, increasing the number of iterations per time instant in this case does not necessarily yield to lower regrets. 

%% file: figures/errors_diff_est.tikz
\begin{tikzpicture}
\begin{axis}[%
width=2.5in,
height=1.8in,
at={(0.758in,0.481in)},
scale only axis,
xmin=0,
xmax=140,
xlabel style={font=\color{white!15!black},scale=0.8},
xlabel={time step $k$ },
ymin=0,
ymax=5.3,
ylabel style={font=\color{white!15!black},scale=0.8},
ylabel={Estimation errors $\Vert e_k \Vert$},
every axis y label/.style={at={(axis description cs:-0.15,.5)},rotate=90,anchor=south,scale=0.8},
axis background/.style={fill=white},
xmajorgrids,
ymajorgrids,
grid style={dotted},
legend style={legend columns=1, column sep=3pt,legend cell align=left, align=left, draw=white!15!black,  font= \scriptsize}
]

     \addplot [each nth point=1, filter discard warning=false, unbounded coords=discard,color=magenta, line width=0.5pt, forget plot]
  table[x=t, y=e_Luen]{Data/e_i1.txt};
  \addplot [each nth point=1, filter discard warning=false, unbounded coords=discard 
,color=magenta, draw=none, mark size=1.7pt, mark=x, mark options={solid,,magenta}]
 table[x=t, y=e_Luen]{Data/e_i1.txt}; 

        \addplot [ each nth point=1, filter discard warning=false, unbounded coords=discard, color=ForestGreen	, line width=0.8pt, forget plot]
  table[x=t, y=e_con]{Data/e_i1.txt};
  \addplot [each nth point=1,color=ForestGreen	, draw=none, mark size=1.3pt, mark=triangle, mark options={solid,ForestGreen}]
  table[x=t, y=e_con]{Data/e_i1.txt};

    \addplot [ each nth point=1, filter discard warning=false, unbounded coords=discard, color=blue	, line width=0.8pt, forget plot]
  table[x=t, y=e_DMD]{Data/e_i1.txt};
  \addplot [each nth point=1,color=blue	, draw=none, mark size=1.3pt, mark=square, mark options={solid,blue}]
  table[x=t, y=e_DMD]{Data/e_i1.txt};

%

    \addplot [ each nth point=1, filter discard warning=false, unbounded coords=discard, color=gray	, line width=0.8pt, forget plot]
  table[x=t, y=e_DMD]{Data/e_i200.txt};
  \addplot [each nth point=1,color=gray	, draw=none, mark size=0.8pt, mark=o, mark options={solid,gray}]
  table[x=t, y=e_DMD]{Data/e_i200.txt};
  


\addlegendentry{Luenberger}
 \addlegendentry{pMHE \cite{gharbi2018proximity}}
 \addlegendentry{anytime pMHE ($\mathrm{it}=1$)}
\addlegendentry{anytime pMHE ($\mathrm{it} =200$)}

%
%
%

\end{axis}
\end{tikzpicture}%

%% file: figures/errorsiT.tikz
\begin{tikzpicture}
\begin{axis}[%
width=2.5in,
height=1.8in,
at={(0.758in,0.481in)},
scale only axis,
xmin=0,
xmax=30,
xlabel style={font=\color{white!15!black},scale=0.8},
xlabel={time step $k$ },
ymin=0,
ymax=5.3,
ylabel style={font=\color{white!15!black},scale=0.8},
ylabel={Estimation errors $\Vert e_k \Vert$},
every axis y label/.style={at={(axis description cs:-0.15,.5)},rotate=90,anchor=south,scale=0.8},
axis background/.style={fill=white},
xmajorgrids,
ymajorgrids,
grid style={dotted},
legend style={legend columns=1, column sep=3pt,legend cell align=left, align=left, draw=white!15!black,  font= \scriptsize}
]

    \addplot [ each nth point=1, filter discard warning=false, unbounded coords=discard, color=blue	, line width=0.8pt, forget plot]
  table[x=t, y=e_DMD]{Data/e_i1.txt};
  \addplot [each nth point=1,color=blue	, draw=none, mark size=1.3pt, mark=square, mark options={solid,blue}]
  table[x=t, y=e_DMD]{Data/e_i1.txt};
     \addplot [each nth point=1, filter discard warning=false, unbounded coords=discard,color=magenta, line width=0.5pt, forget plot]
  table[x=t, y=e_DMD]{Data/e_i2.txt};
  \addplot [each nth point=1, filter discard warning=false, unbounded coords=discard 
,color=magenta, draw=none, mark size=1.7pt, mark=x, mark options={solid,,magenta}]
 table[x=t, y=e_DMD]{Data/e_i2.txt}; 

      \addplot [ each nth point=1, filter discard warning=false, unbounded coords=discard, color=black	, line width=0.8pt, forget plot]
  table[x=t, y=e_DMD]{Data/e_i5.txt};
  \addplot [each nth point=1,color=black	, draw=none, mark size=1.3pt, mark=triangle, mark options={solid,black}]
  table[x=t, y=e_DMD]{Data/e_i5.txt};

        \addplot [ each nth point=1, filter discard warning=false, unbounded coords=discard, color=gray	, line width=0.8pt, forget plot]
  table[x=t, y=e_DMD]{Data/e_i200.txt};
  \addplot [each nth point=1,color=gray	, draw=none, mark size=1.3pt, mark=o, mark options={solid,gray}]
  table[x=t, y=e_DMD]{Data/e_i200.txt};

%

  
 \addlegendentry{$\mathrm{it}=1$}
\addlegendentry{$\mathrm{it}=2$}
  \addlegendentry{$\mathrm{it}=5 $}
\addlegendentry{$\mathrm{it}  =200$}



%
%
%

\end{axis}
\end{tikzpicture}%

%% file: figures/regret_non_aver_Diff_iT_N2.tikz
\begin{tikzpicture}
\begin{axis}[%
width=2.5in,
height=2in,
at={(0.758in,0.481in)},
scale only axis,
xmin=0,
xmax=200,
xlabel style={font=\color{white!15!black},scale=0.8},
xlabel={Simulation time $T$  },
ymode=log,
ylabel style={font=\color{white!15!black},scale=0.8},
ylabel={$R(T)$},
every axis y label/.style={at={(axis description cs:-0.15,.5)},rotate=90,anchor=south,scale=0.8},
axis background/.style={fill=white},
xmajorgrids,
ymajorgrids,
grid style={dotted},
legend style={at={(0.6,0.32)},anchor=south west, legend columns=1, column sep=3pt, align=left, draw=white!15!black,  font= \scriptsize}
]
  \addplot [ each nth point=1, filter discard warning=false, unbounded coords=discard, color=blue	, line width=0.8pt, forget plot]
  table[x=t, y=regret_vec]{Data/T_N_2_iT_1_true.txt};
  \addplot [each nth point=1,color=blue	, draw=none, mark size=1.3pt, mark=square, mark options={solid,blue}]
  table[x=t, y=regret_vec]{Data/T_N_2_iT_1_true.txt};
     \addplot [each nth point=1, filter discard warning=false, unbounded coords=discard,color=magenta, line width=0.5pt, forget plot]
  table[x=t, y=regret_vec]{Data/T_N_2_iT_2_true.txt};
  \addplot [each nth point=1, filter discard warning=false, unbounded coords=discard 
,color=magenta, draw=none, mark size=1.7pt, mark=x, mark options={solid,,magenta}]
 table[x=t, y=regret_vec]{Data/T_N_2_iT_2_true.txt}; 

      \addplot [ each nth point=1, filter discard warning=false, unbounded coords=discard, color=black	, line width=0.8pt, forget plot]
  table[x=t, y=regret_vec]{Data/T_N_2_iT_5_true.txt};
  \addplot [each nth point=1,color=black	, draw=none, mark size=1.3pt, mark=triangle, mark options={solid,black}]
  table[x=t, y=regret_vec]{Data/T_N_2_iT_5_true.txt};
  
        \addplot [ each nth point=1, filter discard warning=false, unbounded coords=discard, color=ForestGreen	, line width=0.8pt, forget plot]
  table[x=t, y=regret_vec]{Data/T_N_2_iT_20_true.txt};
  \addplot [each nth point=1,color=ForestGreen	, draw=none, mark size=1.3pt, mark=triangle, mark options={solid,ForestGreen}]
  table[x=t, y=regret_vec]{Data/T_N_2_iT_20_true.txt};
  
 \addplot [ each nth point=1, filter discard warning=false, unbounded coords=discard, color=gray	, line width=0.8pt, forget plot]
  table[x=t, y=reg_pMHE_vec]{Data/T_N_2_pMHE.txt};
  \addplot [each nth point=1,color=gray	, draw=none, mark size=1.3pt, mark=o, mark options={solid,gray}]
  table[x=t, y=reg_pMHE_vec]{Data/T_N_2_pMHE.txt};
%

%

 \addlegendentry{$\mathrm{it}=1$}
\addlegendentry{$\mathrm{it}=2$}
 \addlegendentry{$\mathrm{it}=5 $}
\addlegendentry{$\mathrm{it}=20 $}
\addlegendentry{pMHE \cite{gharbi2018proximity}}

\end{axis}
\end{tikzpicture}%

%% file: figures/regret_Diff_iT_N2.tikz
\begin{tikzpicture}
\begin{axis}[%
width=2.5in,
height=1.8in,
at={(0.758in,0.481in)},
scale only axis,
xmin=0,
xmax=200,
xlabel style={font=\color{white!15!black},scale=0.8},
xlabel={Simulation time $T$  },
ymode=log,
ylabel style={font=\color{white!15!black},scale=0.8},
ylabel={$R(T) /T$},
every axis y label/.style={at={(axis description cs:-0.15,.5)},rotate=90,anchor=south,scale=0.8},
axis background/.style={fill=white},
xmajorgrids,
ymajorgrids,
grid style={dotted},
legend style={legend columns=1, column sep=3pt,legend cell align=left, align=left, draw=white!15!black,  font= \scriptsize}
]

    \addplot [ each nth point=1, filter discard warning=false, unbounded coords=discard, color=blue	, line width=0.8pt, forget plot]
  table[x=t, y=ave_regret_vec]{Data/T_N_2_iT_1_true.txt};
  \addplot [each nth point=1,color=blue	, draw=none, mark size=1.3pt, mark=square, mark options={solid,blue}]
  table[x=t, y=ave_regret_vec]{Data/T_N_2_iT_1_true.txt};
     \addplot [each nth point=1, filter discard warning=false, unbounded coords=discard,color=magenta, line width=0.5pt, forget plot]
  table[x=t, y=ave_regret_vec]{Data/T_N_2_iT_2_true.txt};
  \addplot [each nth point=1, filter discard warning=false, unbounded coords=discard 
,color=magenta, draw=none, mark size=1.7pt, mark=x, mark options={solid,,magenta}]
 table[x=t, y=ave_regret_vec]{Data/T_N_2_iT_2_true.txt}; 

      \addplot [ each nth point=1, filter discard warning=false, unbounded coords=discard, color=black	, line width=0.8pt, forget plot]
  table[x=t, y=ave_regret_vec]{Data/T_N_2_iT_5_true.txt};
  \addplot [each nth point=1,color=black	, draw=none, mark size=1.3pt, mark=triangle, mark options={solid,black}]
  table[x=t, y=ave_regret_vec]{Data/T_N_2_iT_5_true.txt};
        \addplot [ each nth point=1, filter discard warning=false, unbounded coords=discard, color=ForestGreen	, line width=0.8pt, forget plot]
  table[x=t, y=ave_regret_vec]{Data/T_N_2_iT_20_true.txt};
  \addplot [each nth point=1,color=ForestGreen	, draw=none, mark size=1.3pt, mark=triangle, mark options={solid,ForestGreen}]
  table[x=t, y=ave_regret_vec]{Data/T_N_2_iT_20_true.txt};

 \addplot [ each nth point=1, filter discard warning=false, unbounded coords=discard, color=gray	, line width=0.8pt, forget plot]
  table[x=t, y=ave_regret_pMHE_vec]{Data/T_N_2_pMHE.txt};
  \addplot [each nth point=1,color=gray	, draw=none, mark size=1.3pt, mark=o, mark options={solid,gray}]
  table[x=t, y=ave_regret_pMHE_vec]{Data/T_N_2_pMHE.txt};


  
 \addlegendentry{$\mathrm{it}=1$}
\addlegendentry{$\mathrm{it}=2$}
 \addlegendentry{$\mathrm{it}=5 $}
\addlegendentry{$\mathrm{it}=20 $}
\addlegendentry{pMHE \cite{gharbi2018proximity}}


%
%
%

\end{axis}
\end{tikzpicture}%

%% file: figures/regret_non_aver_Diff_iT_N2_warmConst.tikz
\begin{tikzpicture}
\begin{axis}[%
width=2.5in,
height=1.8in,
at={(0.758in,0.481in)},
scale only axis,
xmin=0,
xmax=200,
xlabel style={font=\color{white!15!black},scale=0.8},
xlabel={Simulation time $T$  },
ylabel style={font=\color{white!15!black},scale=0.8},
ylabel={$R(T)$},
every axis y label/.style={at={(axis description cs:-0.15,.5)},rotate=90,anchor=north,scale=0.8},
axis background/.style={fill=white},
xmajorgrids,
ymajorgrids,
grid style={dotted},
legend style={legend pos=south east, legend columns=1, column sep=3pt, legend cell align=left,  draw=white!15!black,  font= \scriptsize}
]
  \addplot [ each nth point=1, filter discard warning=false, unbounded coords=discard, color=blue	, line width=0.8pt, forget plot]
  table[x=t, y=regret_vec]{Data/Reg_warm_N2_iT_1_true.txt};
  \addplot [each nth point=1,color=blue	, draw=none, mark size=1.3pt, mark=square, mark options={solid,blue}]
  table[x=t, y=regret_vec]{Data/Reg_warm_N2_iT_1_true.txt};
     \addplot [each nth point=1, filter discard warning=false, unbounded coords=discard,color=magenta, line width=0.5pt, forget plot]
  table[x=t, y=regret_vec]{Data/Reg_warm_N2_iT_2_true.txt};
  \addplot [each nth point=1, filter discard warning=false, unbounded coords=discard 
,color=magenta, draw=none, mark size=1.7pt, mark=x, mark options={solid,,magenta}]
 table[x=t, y=regret_vec]{Data/Reg_warm_N2_iT_2_true.txt}; 

      \addplot [ each nth point=1, filter discard warning=false, unbounded coords=discard, color=black	, line width=1.8pt, forget plot]
  table[x=t, y=regret_vec]{Data/Reg_warm_N2_iT_5_true.txt};
  \addplot [each nth point=1,color=black	, draw=none, mark size=1.8pt, mark=o, mark options={solid,black}]
  table[x=t, y=regret_vec]{Data/Reg_warm_N2_iT_5_true.txt};
  
        \addplot [ each nth point=1, filter discard warning=false, unbounded coords=discard, color=ForestGreen	, line width=0.8pt, forget plot]
  table[x=t, y=regret_vec]{Data/Reg_warm_N2_iT_20_true.txt};
  \addplot [each nth point=1,color=ForestGreen	, draw=none, mark size=1.3pt, mark=triangle, mark options={solid,ForestGreen}]
  table[x=t, y=regret_vec]{Data/Reg_warm_N2_iT_20_true.txt};
%
  
 \addlegendentry{$\mathrm{it}=1$}
\addlegendentry{$\mathrm{it}=2$}
 \addlegendentry{$\mathrm{it}=5 $}
\addlegendentry{$\mathrm{it}=20 $}
%


\end{axis}
\end{tikzpicture}%

%% file: conc.tex
\section{Conclusion}
\label{sec: conclusion}
In this paper, we presented a computationally tractable approach for constrained MHE of discrete-time linear systems. An anytime pMHE iteration scheme is proposed in which a state estimate at each time instant is computed based on an arbitrary number of optimization algorithm iterations. The underlying optimization algorithm  consists of a mirror descent-like method which generalizes the gradient descent and can therefore be executed quickly. Under suitable assumptions on the Bregman distance and the step sizes, GES of the estimation errors was  established and is ensured after any  number of optimization
algorithm iterations. In addition, the performance of the iteration scheme was characterized by the resulting real-time regret for which upper bounds were derived. The proposed iteration scheme provides stable estimates after each optimization iteration and possesses a sublinear regret which can be rendered 
arbitrarily small by increasing the number of iterations.\\
The proposed anytime pMHE iteration scheme is conceptually related to the anytime model predictive control (MPC) iteration scheme with relaxed barrier
functions~\cite{feller2017stabilizing,feller2018sparsity}, where stabilizing control inputs are generated  after any number of optimization iterations. Our goal in future research is to combine both the MPC and MHE iteration schemes in an overall anytime
estimation-based MPC algorithm. Furthermore, comparisons to  real-time MHE techniques established in the literature and a further exploration of the computational complexity of the proposed algorithm deserve  further research. Moreover, it would be interesting to study the robustness properties of the iteration scheme with respect to process and measurement disturbances.   

%% file: App.tex
\subsection{Reformulation of the estimation problem}
\label{app:condensed}
Using the system dynamics (\ref{dnamics}) and (\ref{dnamicsv}), we can write each output residual $\hat{v}_i$ in the estimation window in terms of the decision variable $\mathbf{\hat{z}}_k$ defined in (\ref{eq:z-def}) as follows
\begin{align}
  \hat{v}_i &=  y_i - C \hat{x_i}  \\
  &= y_i - \mathcal{O}_i\, \hat{x}_{k-N} -C \tilde{u}_i -\sum_{j=k-N}^{i-1} \!C A^{i-j-1} \hat{w}_j \nonumber
\end{align}
with $\mathcal{O}_i \coloneqq CA^{i-k+N} $ and $
\tilde{u}_i\coloneqq \sum\limits_{j=k-N}^{i-1} A^{i-j-1} B u_j$. We obtain the sum of stage costs
\begin{align} \label{sumStage}
  f_k\left( \mathbf{\hat{z}}_k \right)& \coloneqq \sum_{i=k-N}^{k-1}  q\left(\hat{w}_i \right) \\ & \quad + r\Big( y_i - \mathcal{O}_i\, \hat{x}_{k-N} -C \tilde{u}_i -\sum_{j=k-N}^{i-1} \!C A^{i-j-1} \hat{w}_j\Big) \nonumber.
\end{align}
The matrices $G$ and $F$ in the constraint set $\mathcal{S}_k$ defined in (\ref{stackedConstr}) are given by
\begin{align}
G&\coloneqq \begin{bmatrix} C_{\text{x}}\\C_{\text{x}}A \\ \vdots \\ C_{\text{x}}A^{N} \end{bmatrix} \in \mathbb{R}^{(N+1)q_{\text{x}} \times n}, \nonumber \\
 F &\coloneqq \begin{bmatrix}
0 & 0 &\hdots & 0\\
C_{\text{x}} & 0 &\hdots & 0\\
C_{\text{x}}A & C_{\text{x}} &\hdots & 0\\
\vdots &\vdots &\hdots & \vdots  \\
C_{\text{x}}A^{N-1}\, & C_{\text{x}}A^{N-2}\,&\hdots  & C_{\text{x}} 
\end{bmatrix} \in \mathbb{R}^{(N+1)q_{\text{x}}\times Nn}
\end{align}
and the vector $E_k$ is 
\begin{align}
 E_k&\coloneqq \begin{bmatrix} d_{\text{x}}\\d_{\text{x}}-C_{\text{x}}\tilde{u}_{k-N+1} \\ \vdots \\ d_{\text{x}}-C_{\text{x}}\tilde{u}_{k} \end{bmatrix}  \in \mathbb{R}^{(N+1)q_{\text{x}}}.
\end{align}
\subsection{Bregman distances}
\label{app: Bregman}
We shortly present central properties of Bregman distances defined in (\ref{BregmanDef}). Given the strong convexity of $\psi$, it follows that $D_\psi(\mathbf{z}_1,\mathbf{z}_2)$ is nonngeative, and that $D_\psi(\mathbf{z}_1,\mathbf{z}_2)=0$ if and only if $\mathbf{z}_1=\mathbf{z}_2$. Moreover, if  $\psi(\mathbf{z})=\frac{1}{2} \Vert \mathbf{z} \Vert^2  $, we obtain $D_\psi(\mathbf{z}_1,\mathbf{z}_2)= \frac{1}{2}\Vert  \mathbf{z}_1 - ~\mathbf{z}_2 \Vert^2 $, which is the quadratic Euclidean distance. In analogy with the classical projection, the Bregman projection $\Pi_{\mathcal{S}}^{\psi} (\mathbf{\bar{z}})$ onto a convex set $\mathcal{S}$ is defined as the closest point in $\mathcal{S}$  to $\mathbf{\bar{z}}$ with respect to the Bregman distance $D_\psi$:
\begin{align}
\label{BregProjection}
    \Pi_{\mathcal{S}}^{\psi} (\mathbf{\bar{z}})= \argmin\limits_{\mathbf{z}\in \mathcal{S}} \quad  D_\psi (\mathbf{z},\mathbf{\bar{z}}).
\end{align}
The next key identity can be proven by directly using the definition of $D_\psi$.
 \begin{lem} 
Let the function $D_{\psi}$ denote a  Bregman distance induced from $\psi$. Then for any $a,b,c \in  \mathbb{R}^{(N+1)n} $, the following three-points identity holds
 \begin{align}
 \label{threePoint}
 D_\psi (c,a)+&D_\psi (a,b) - D_\psi (c,b)  \\&= \left(\nabla \psi (b) - \nabla \psi (a) \right)^\top (c-a). \nonumber
 \end{align}
 \end{lem}
 We require the next result from~\cite[Proposition~3.5]{censor1992proximal}.
 \begin{lem}
 \label{lem:bregmanProj}
 Let the set $\mathcal{S} \subset \mathbb{R}^{(N+1)n}$ be nonempty, closed and convex. Suppose $\bar{\mathbf{z}} \notin \mathcal{S} $ and   $\mathbf{z} \in \mathcal{S}$. Then, 
 \begin{align}
  \label{bregmanProj}
     D_\psi \big(\Pi_{\mathcal{S}}^{\psi} (\mathbf{\bar{z}}), \bar{\mathbf{z}} \big) \leq D_\psi \big( \mathbf{z} ,  \bar{\mathbf{z}}  \big) - D_\psi \big( \mathbf{z}, \Pi_{\mathcal{S}}^{\psi} (\mathbf{\bar{z}})\big) .
 \end{align}
 \end{lem}
  For more details on Bregman distances, we refer the reader to \cite{censor1992proximal}.
  \subsection{Proof of Lemma \ref{lem: Stabadap}}
  \label{app:stabilityLemma}
   \begin{proof}
   The proof generalizes and follows similar steps as in the proof of  \cite[Proposition 2]{mokhtari2016online}, in which the performance of the online gradient descent method is investigated.  
   Convexity of $f_k$ implies that
   \begin{align}
   f_k(\mathbf{z})  \geq f_k \left( \mathbf{\hat{z}}_{k}^i \right) + \nabla f_k \left( \mathbf{\hat{z}}_{k}^i\right)^\top \left(\mathbf{z} -  \mathbf{\hat{z}}_{k}^i\right)
   \end{align}
   for any $\mathbf{z} \in \mathcal{S}_k$ and hence 
      \begin{align}
      \label{inter1}
   f_k(\mathbf{z})  \geq f_k \left( \mathbf{\hat{z}}_{k}^i \right) &+ \nabla f_k \left( \mathbf{\hat{z}}_{k}^i \right)^\top \left(\mathbf{\hat{z}}_{k}^{i+1} - \mathbf{\hat{z}}_{k}^i\right) \\&  + \nabla f_k \left( \mathbf{\hat{z}}_{k}^i\right)^\top \left( \mathbf{z} -\mathbf{\hat{z}}_{k}^{i+1}  \right). \nonumber
   \end{align}
   By optimality of $\mathbf{\hat{z}}_{k}^{i+1} $ in (\ref{optAdap})
   and by \eqref{BregmanDef},
   we have for any $\mathbf{z} \in \mathcal{S}_k$ 
\begin{align}
\label{toUseinNextLem}
\hspace{-0.5cm}\big( \eta_k^i  \nabla f_k\left(\mathbf{\hat{z}}_{k}^i  \right) \! +\! \nabla \psi\big(\mathbf{\hat{z}}_{k}^{i+1} \big) \!- \!  \nabla \psi \big(\mathbf{\hat{z}}_{k}^i \big) \big)^\top\!\! \left(\mathbf{z} -\mathbf{\hat{z}}_{k}^{i+1}  \right) \geq 0.  
\end{align}
Thus (\ref{inter1}) becomes 
 \begin{align}
 \label{inter2}
   f_k(\mathbf{z})  \geq f_k \left( \mathbf{\hat{z}}_{k}^i \right) & + \nabla f_k \left(\mathbf{\hat{z}}_{k}^i \right)^\top \left(\mathbf{\hat{z}}_{k}^{i+1} - \mathbf{\hat{z}}_{k}^i\right)  \\ &+ \frac{1}{\eta_k^i} \big(  \nabla \psi \big(\mathbf{\hat{z}}_{k}^i\big) - \nabla \psi\big(\mathbf{\hat{z}}_{k}^{i+1} \big)    \big)^\top\, \left(\mathbf{z}-\mathbf{\hat{z}}_{k}^{i+1}  \right) .  \nonumber
   \end{align}
   Since the gradients of $f_k$ are Lipschitz continuous by  Assumption \ref{ass: f strngSmooth}, we have that
   \begin{align}
   f_k \left(\mathbf{\hat{z}}_{k}^{i+1}  \right) \leq   f_k \left( \mathbf{\hat{z}}_{k}^i \right)& + \nabla f_k \left( \mathbf{\hat{z}}_{k}^i \right) ^\top \left(\mathbf{\hat{z}}_{k}^{i+1}-  \mathbf{\hat{z}}_{k}^i\right) \\& + \frac{L_f}{2} \left\Vert \mathbf{\hat{z}}_{k}^{i+1}- \mathbf{\hat{z}}_{k}^i \right\Vert^2,  \nonumber 
   \end{align}
and hence
\begin{align}
    \label{inter3}
  f_k(\mathbf{z}) \geq  f_k \left(\mathbf{\hat{z}}_{k}^{i+1}   \right)& - \frac{L_f}{2} \left\Vert \mathbf{\hat{z}}_{k}^{i+1} -  \mathbf{\hat{z}}_{k}^i \right\Vert^2  \\&\!\!\! +  \frac{1}{\eta_k^i} \big(  \nabla \psi \big(\mathbf{\hat{z}}_{k}^i \big) - \nabla \psi\big( \mathbf{\hat{z}}_{k}^{i+1}  \big)    \big)^\top\, \left(\mathbf{z}-\mathbf{\hat{z}}_{k}^{i+1}  \right).   \nonumber 
   \end{align}
   In view of the three points identity (\ref{threePoint}) and the strong convexity of $D_\psi$, we have
    \begin{align}
\big(  \nabla \psi \big( \mathbf{\hat{z}}_{k}^i\big) &- \nabla \psi\big(\mathbf{\hat{z}}_{k}^{i+1}  \big)    \big)^\top\, \left(\mathbf{z}-\mathbf{\hat{z}}_{k}^{i+1}   \right)\\ \nonumber &= D_\psi \big(\mathbf{z}, \mathbf{\hat{z}}_{k}^{i+1}  \big)+D_\psi \big( \mathbf{\hat{z}}_{k}^{i+1}  ,\mathbf{\hat{z}}_{k}^i\big) - D_\psi \big(\mathbf{z},\mathbf{\hat{z}}_{k}^i\big) \\\nonumber 
& \geq D_\psi \big(\mathbf{z}, \mathbf{\hat{z}}_{k}^{i+1}  \big)+\frac{\sigma}{2} \Vert  \mathbf{\hat{z}}_{k}^{i+1}   - \mathbf{\hat{z}}_{k}^i \Vert^2 - D_\psi \big(\mathbf{z},\mathbf{\hat{z}}_{k}^i\big) .
 \end{align}
Therefore, using (\ref{inter3})
    \begin{align}
  f_k(\mathbf{z})& \geq  f_k \left(\mathbf{\hat{z}}_{k}^{i+1}   \right) - \frac{L_f}{2} \left\Vert \mathbf{\hat{z}}_{k}^{i+1} -  \mathbf{\hat{z}}_{k}^i \right\Vert^2  \nonumber \\&   \enskip+  \frac{1}{\eta_k^i} \big( D_\psi \big(\mathbf{z}, \mathbf{\hat{z}}_{k}^{i+1}  \big)- D_\psi \big(\mathbf{z},\mathbf{\hat{z}}_{k}^i\big) + \frac{\sigma}{2} \Vert   \mathbf{\hat{z}}_{k}^{i+1}   - \mathbf{\hat{z}}_{k}^i \Vert^2 \big)    \nonumber  
  \\  &=    f_k \left(\mathbf{\hat{z}}_{k}^{i+1}   \right) + \frac{1}{2}  \big( \frac{\sigma}{\eta_k^i} - L_f \big) \left\Vert  \mathbf{\hat{z}}_{k}^{i+1}   - \mathbf{\hat{z}}_{k}^i \right\Vert^2   \label{LipschitzForRegret} \\& \enskip +  \frac{1}{\eta_k^i} \big( D_\psi \big(\mathbf{z}, \mathbf{\hat{z}}_{k}^{i+1}  \big)- D_\psi \big(\mathbf{z},\mathbf{\hat{z}}_{k}^i\big)  \big) .  \nonumber
   \end{align}
   We set $\mathbf{z}=\mathbf{z}_k \in \mathcal{S}_k$, i.e. the true state with zero model residual, and obtain
       \begin{align}
       \label{toUseInRegret}
\!\!\!0\geq  f_k( \mathbf{z}_k) - f_k \left(\mathbf{\hat{z}}_{k}^{i+1}   \right)\!   & \geq      \frac{1}{2}  \big( \frac{\sigma}{\eta_k^i} - L_f \big) \left\Vert\mathbf{\hat{z}}_{k}^{i+1}   - \mathbf{\hat{z}}_{k}^i \right\Vert^2 \\& +  \frac{1}{\eta_k^i} \big(  D_\psi \big(\mathbf{z}_k, \mathbf{\hat{z}}_{k}^{i+1}  \big)- D_\psi \big(\mathbf{z}_k,\mathbf{\hat{z}}_{k}^i\big) \big) \nonumber .   
   \end{align}
The inequality $ f_k \left(\mathbf{\hat{z}}_{k}^{i+1}  \right) \geq  f_k( \mathbf{z}_k)    $ holds by  Assumption~\ref{ass: f convex},  which states that $f_k$ achieves its minimal value at $\mathbf{z}_k $.  Hence, we get 
        \begin{align}
        \label{betweenIt}
D_\psi \big(\mathbf{z}_k, \mathbf{\hat{z}}_{k}^{i+1}  \big) & \leq D_\psi \big(\mathbf{z}_k,\mathbf{\hat{z}}_{k}^i\big) +      \frac{\eta_k^i}{2}  \big(L_f -  \frac{\sigma}{\eta_k^i}  \big) \left\Vert \mathbf{\hat{z}}_{k}^{i+1}   - \mathbf{\hat{z}}_{k}^i \right\Vert^2,
   \end{align}
which proves the first statement in Lemma  \ref{lem: Stabadap}. Applying (\ref{betweenIt}) for each  two subsequent iterations $i$ and $i+1$ (where $i=~0,\cdots,j$)  yields
  \begin{align}
  \label{resultwithwarm}
&D_\psi(\mathbf{z}_k, \mathbf{\hat{z}}_{k}^{j})\\ &\leq D_\psi(\mathbf{z}_k,  \mathbf{\hat{z}}_{k}^{j-1} ) \!+\!  \frac{1 }{2}  \big(\eta_k^{j-1} \,L_f - \sigma  \big) \!\left\Vert  \mathbf{\hat{z}}_{k}^{j} -  \mathbf{\hat{z}}_{k}^{j-1}  \right\Vert^2 \nonumber \\
&\leq D_\psi(\mathbf{z}_k,  \mathbf{\hat{z}}_{k}^{j-2} )\! +\! \frac{1 }{2}  \big(\eta_k^{j-1} \,L_f - \sigma  \big) \! \left\Vert \mathbf{\hat{z}}_{k}^{j} -\mathbf{\hat{z}}_{k}^{j-1}  \right\Vert^2  \nonumber \\ &\qquad +    \frac{1 }{2}  \big(\eta_k^{j-2} \,L_f - \sigma  \big)\left\Vert  \mathbf{\hat{z}}_{k}^{j-1 } -\mathbf{\hat{z}}_{k}^{j-2}  \right\Vert^2 \nonumber \\
&\leq  \cdots \nonumber \\
&\leq D_\psi(\mathbf{z}_k, \mathbf{\hat{z}}_{k}^{0}) +  \frac{1 }{2}  \sum_{i=0}^{j-1} \big(\eta_k^{i} \,L_f - \sigma  \big)\left\Vert  \mathbf{\hat{z}}_{k}^{i+1} -\mathbf{\hat{z}}_{k}^{i}  \right\Vert^2 \nonumber.
\end{align}
   Since  $\mathbf{\hat{z}}_{k}^0= \Pi_{\mathcal{S}_k}^{\psi}\left(\mathbf{\bar{z}}_{k} \right)$  by \eqref{BregProjection} 
   and   $\mathbf{z}_k \in \mathcal{S}_k$, in view of   (\ref{bregmanProj}) in Lemma \ref{lem:bregmanProj}, we have 
   \begin{align}
 0 \leq  D_\psi \left(\mathbf{\hat{z}}_{k}^0, \mathbf{\bar{z}}_{k}\right) \leq D_\psi \left(\mathbf{z}_k , \mathbf{\bar{z}}_{k} \right) - D_\psi \left(\mathbf{z}_k , \mathbf{\hat{z}}_{k}^0 \right).
   \end{align}
   Thus, $D_\psi \left(\mathbf{z}_k , \mathbf{\hat{z}}_{k}^0 \right) \leq D_\psi \left(\mathbf{z}_k , \mathbf{\bar{z}}_{k} \right) $ and we obtain in (\ref{resultwithwarm})  
  \begin{align}
&D_\psi(\mathbf{z}_k, \mathbf{\hat{z}}_{k}^{j})\\ 
&\leq D_\psi \left(\mathbf{z}_k , \mathbf{\bar{z}}_{k} \right)+  \frac{1 }{2}  \sum_{i=0}^{j-1} \big(\eta_k^{i} \,L_f - \sigma  \big)\left\Vert  \mathbf{\hat{z}}_{k}^{i+1} -\mathbf{\hat{z}}_{k}^{i}  \right\Vert^2 \nonumber.
\end{align} 
   \end{proof}
   \subsection{Proof of Lemma \ref{regret_betweeniT}}
   \label{proof: Lemma  regret_betweeniT}
 \begin{proof}
 The following analysis is based on the convergence proof of the mirror descent algorithm  presented in \cite{beck2003mirror}.
Since  $   \mathbf{z}_{k}^c\in \mathcal{S}_k$, we can evaluate the optimality condition  (\ref{toUseinNextLem}) of $\mathbf{\hat{z}}_{k}^{i+1} $  for $ \mathbf{z}= \mathbf{z}_{k}^c$ to obtain
\begin{align}
\label{optxbar}
\big( \eta_k^i  \nabla f_k\left(\mathbf{\hat{z}}_{k}^i  \right)  + \nabla \psi\big(\mathbf{\hat{z}}_{k}^{i+1} \big) -   \nabla \psi \big(\mathbf{\hat{z}}_{k}^i \big) \big)^\top\, \left(\mathbf{z}_{k}^c -\mathbf{\hat{z}}_{k}^{i+1}  \right) \geq 0.    
\end{align}
Given that $f_k$ is convex, we have 
\begin{subequations}
\label{overall} 
\begin{align}
  \eta_k^i \left(f_k(\mathbf{\hat{z}}_{k}^i) - f_k( \mathbf{z}_{k}^c) \right) &\leq \eta_k^i\nabla f_k\left(\mathbf{\hat{z}}_{k}^i \right)^\top (\mathbf{\hat{z}}_{k}^i -\mathbf{z}_{k}^c )   \\
&=  s_1 +s_2 +s_3,   \nonumber
\end{align}
where 
\begin{align}
s_1 &\coloneqq \big(   \nabla \psi \big(\mathbf{\hat{z}}_{k}^i\big) -\nabla \psi\big(\mathbf{\hat{z}}_{k}^{i+1}\big) -\eta_k^i  \nabla f_k\left(\mathbf{\hat{z}}_{k}^{i} \right) \big)^\top\, \left(\mathbf{z}_{k}^c-\mathbf{\hat{z}}_{k}^{i+1} \right)  \label{rh1}   \\
s_2 &\coloneqq   \big(   \nabla \psi\big(\mathbf{\hat{z}}_{k}^{i+1}\big) - \nabla \psi \big(\mathbf{\hat{z}}_{k}^{i}\big)     \big)^\top\, \left( \mathbf{z}_{k}^c-\mathbf{\hat{z}}_{k}^{i+1}  \right) \label{rh2}    \\
s_3 &\coloneqq \eta_k^i \nabla f_k\left(\mathbf{\hat{z}}_{k}^{i} \right)^\top (\mathbf{\hat{z}}_{k}^{i}  -\mathbf{\hat{z}}_{k}^{i+1} ) \label{rh3}.
\end{align}
\end{subequations}
 By (\ref{optxbar}),  $s_1 \leq 0$. 
Using the three-points identity (\ref{threePoint}) as well as the strong convexity of $D_\psi$ assumed in  Assumption \ref{ass: Bregman}, we have that 
\begin{align} 
\label{inter11}
s_2&= - \big(    \nabla \psi \big(\mathbf{\hat{z}}_{k}^{i}\big)  -  \nabla \psi\big(\mathbf{\hat{z}}_{k}^{i+1}\big)   \big)^\top\, \left(\mathbf{z}_{k}^c-\mathbf{\hat{z}}_{k}^{i+1} \right)
 \\& = D_\psi(\mathbf{z}_{k}^c,\mathbf{\hat{z}}_{k}^{i}) - D_\psi( \mathbf{z}_{k}^c, \mathbf{\hat{z}}_{k}^{i+1}) - D_\psi(\mathbf{\hat{z}}_{k}^{i+1}, \mathbf{\hat{z}}_{k}^{i}) \nonumber\\
& \leq D_\psi( \mathbf{z}_{k}^c,\mathbf{\hat{z}}_{k}^{i}) - D_\psi(\mathbf{z}_{k}^c, \mathbf{\hat{z}}_{k}^{i+1}) - \frac{\sigma}{2} \Vert \mathbf{\hat{z}}_{k}^{i+1} - \mathbf{\hat{z}}_{k}^{i}\Vert^2 \nonumber. 
\end{align}
Moreover, by Young's inequality,
\begin{align}
\label{inter12}
s_3 \leq \eta_k^i  \left( \frac{ \eta_k^i}{2 \sigma} \big\Vert \nabla f_k\left(\mathbf{\hat{z}}_{k}^{i} \right) \big\Vert^2 + \frac{\sigma}{2  \eta_k^i} \big\Vert  \mathbf{\hat{z}}_{k}^{i+1} -\mathbf{\hat{z}}_{k}^{i} \big\Vert^2 \right).
\end{align}
Hence, we obtain the  first statement  of  Lemma \ref{regret_betweeniT} by substituting (\ref{inter11}) and (\ref{inter12}) into (\ref{overall}).
  Evaluating (\ref{result: regret_betweeniT}) for $i=0$ yields  
 \begin{subequations}
 \label{instaFirst}
 \begin{align}
  \eta_k^0 &\left(f_k(\mathbf{\hat{z}}_{k}^0) - f_k( \mathbf{z}_{k}^c) \right) \nonumber\\ &\leq  D_\psi( \mathbf{z}_{k}^c,\mathbf{\hat{z}}_{k}^{0}) - D_\psi(\mathbf{z}_{k}^c, \mathbf{\hat{z}}_{k}^{1}) +   \frac{ ( \eta_k^0)^2}{2 \sigma} \Vert \nabla f_k\left(\mathbf{\hat{z}}_{k}^{0}\right) \Vert^2    
    \\ &=  D_\psi \left(  \mathbf{z}_{k}^c , \mathbf{\hat{z}}_k^0 \right) - D_\psi \left( \mathbf{z}_{k+1}^c ,   \mathbf{\hat{z}}_{k+1}^0\right) +   \frac{ ( \eta_k^0)^2}{2 \sigma} \Vert \nabla f_k\left(\mathbf{\hat{z}}_{k}^{0}\right) \Vert^2  \nonumber\\
   & \quad + T_1 + T_2 , \nonumber
  \end{align}
  where 
  \begin{align}
     T_1 &\coloneqq  D_\psi  \left(  \Phi_k\left(  \mathbf{z}_{k}^c \right)  ,  \mathbf{\bar{z}}_{k+1}  \right)   - D_\psi  \left( \mathbf{z}_{k}^c ,    \mathbf{\hat{z}}_k^1 \right)    \label{newrh3} \\
     T_2 &\coloneqq  D_\psi \left( \mathbf{z}_{k+1}^c ,     \mathbf{\hat{z}}_{k+1}^0 \right)  -D_\psi  \left(   \Phi_k\left( \mathbf{z}_{k}^c \right)   ,     \mathbf{\bar{z}}_{k+1}  \right)    \label{newrh4}.
  \end{align}
  \end{subequations}
  We can compute an upper bound for each of these terms as follows.    
Since $\mathbf{\bar{z}}_{k+1} =\Phi_k\big(  \mathbf{\hat{z}}_k^{\mathrm{it}(k)} \big)$, using (\ref{Dcontadap}), we have
 \begin{align}
 T_1& =        D_\psi  \big( \Phi_k\big(  \mathbf{z}_{k}^c \big)   ,  \Phi_k\big(  \mathbf{\hat{z}}_k^{\mathrm{it}(k)} \big)\big)   - D_\psi  \big( \mathbf{z}_{k}^c ,    \mathbf{\hat{z}}_k^1 \big)  \\
 & \leq    D_\psi  \big(      \mathbf{z}_{k}^c    ,    \mathbf{\hat{z}}_k^{\mathrm{it}(k)} \big)     -D_\psi  \big( \mathbf{z}_{k}^c ,    \mathbf{\hat{z}}_k^1 \big) \nonumber.
 \end{align}
 Moreover, employing (\ref{result: regret_betweeniT}) in Lemma \ref{regret_betweeniT} (as just proved above) for each iteration step  starting from $i=\mathrm{it}(k)-1$  to $i=2$ yields
\begin{align}
\label{toSubstitue1}
T_1&  \leq    
      D_\psi  \big(   \mathbf{z}_{k}^c    ,    \mathbf{\hat{z}}_k^{\mathrm{it}(k) -1}  \big)   +      \frac{  (\eta_k^{\mathrm{it}(k)-1})^2 }{2 \sigma} \big\Vert \nabla f_k \big( \mathbf{\hat{z}}_k^{\mathrm{it}(k) -1} \big) \big\Vert^2 \nonumber   \\
& \enskip +   \eta_k^{\mathrm{it}(k)-1} \big( f_k( \mathbf{z}_{k}^c) -f_k\big(\mathbf{\hat{z}}_k^{\mathrm{it}(k) -1} \big)  \big)   -D_\psi  \left( \mathbf{z}_{k}^c ,    \mathbf{\hat{z}}_k^1 \right) \nonumber\\
  &\leq   D_\psi  \big(  \mathbf{z}_{k}^c     ,    \mathbf{\hat{z}}_k^{\mathrm{it}(k)-2}  \big)   +  \frac{  (\eta_k^{\mathrm{it}(k)-2})^2 }{2 \sigma} \big\Vert \nabla f_k \big( \mathbf{\hat{z}}_k^{\mathrm{it}(k)-2} \big) \big\Vert^2  \nonumber  \\
& \enskip +    \frac{  (\eta_k^{\mathrm{it}(k)-1})^2 }{2 \sigma} \big\Vert \nabla f_k \big(  \mathbf{\hat{z}}_k^{\mathrm{it}(k)-1} \big) \big\Vert^2  \nonumber \\& \enskip +   \eta_k^{\mathrm{it}(k)-2} \big( f_k( \mathbf{z}_{k}^c ) -f_k\big( \mathbf{\hat{z}}_k^{\mathrm{it}(k)-2} \big)  \big) \nonumber \\& \enskip +  \eta_k^{\mathrm{it}(k)-1} \big( f_k(  \mathbf{z}_{k}^c ) -f_k\big( \mathbf{\hat{z}}_k^{\mathrm{it}(k)-1} \big)  \big)   -D_\psi  \left( \mathbf{z}_{k}^c ,    \mathbf{\hat{z}}_k^1 \right) \nonumber \\
  & \leq \dots  \\
    &\leq  \!\!   \sum_{i=1}^{\mathrm{it}(k)-1}   \frac{  (\eta_k^{i})^2 }{2 \sigma} \Vert \nabla f_k \left( \mathbf{\hat{z}}_k^{i} \right) \Vert^2        \! + \! \!\sum_{i=1}^{\mathrm{it}(k)-1} \! \eta_k^{i}   \left( f_k(  \mathbf{z}_{k}^c )  -    f_k(\mathbf{\hat{z}}_k^{i}) \right) \nonumber  .  
 \end{align}
 Given that $\mathbf{\hat{z}}_{k}^0= \Pi_{\mathcal{S}_k}^{\psi}\left(\mathbf{\bar{z}}_{k} \right)$  and   $\mathbf{z}_k^c \in \mathcal{S}_k$,   by (\ref{bregmanProj}) in Lemma~\ref{lem:bregmanProj}, we have
   \begin{align}
 0 \leq  D_\psi \left(\mathbf{\hat{z}}_{k}^0, \mathbf{\bar{z}}_{k}\right) \leq D_\psi \left(\mathbf{z}_k^c , \mathbf{\bar{z}}_{k} \right) - D_\psi \left(\mathbf{z}_k^c , \mathbf{\hat{z}}_{k}^0 \right),
   \end{align}
 for all $k>0$.  Hence, $D_\psi \left(\mathbf{z}_k^c , \mathbf{\hat{z}}_{k}^0 \right) \leq D_\psi \left(\mathbf{z}_k^c , \mathbf{\bar{z}}_{k} \right) $, for all $k>0$, and we obtain 
    \begin{align}
 T_2&= D_\psi \left( \mathbf{z}_{k+1}^c ,     \mathbf{\hat{z}}_{k+1}^0 \right)  -D_\psi  \left(   \Phi_k\left( \mathbf{z}_{k}^c \right)   ,     \mathbf{\bar{z}}_{k+1}  \right)
 \\& \leq D_\psi \left( \mathbf{z}_{k+1}^c ,     \mathbf{\bar{z}}_{k+1} \right)  -D_\psi  \left( \Phi_k\left( \mathbf{z}_{k}^c \right)   ,      \mathbf{\bar{z}}_{k+1} \right)  \nonumber .
 \end{align}
 In addition, using the definition of the Bregman distance and the convexity of $\psi$, we get 
 \begin{align}
 \label{toSubstitue2}
 T_2& \leq  \psi(\mathbf{z}_{k+1}^c )   - \nabla \psi(\mathbf{\bar{z}}_{k+1})^\top (  \mathbf{z}_{k+1}^c  - \mathbf{\bar{z}}_{k+1}) -\psi(\Phi_k\left( \mathbf{z}_{k}^c  \right) ) \nonumber  \\
 & \qquad + \nabla \psi(  \mathbf{\bar{z}}_{k+1})^\top ( \Phi_k\left( \mathbf{z}_{k}^c \right)  -  \mathbf{\bar{z}}_{k+1} )  \nonumber \\
 &= \psi(\mathbf{z}_{k+1}^c ) -\psi(\Phi_k\left( \mathbf{z}_{k}^c  \right) ) - \nabla \psi(\mathbf{\bar{z}}_{k+1})^\top ( \mathbf{z}_{k+1}^c  - \Phi_k\left( \mathbf{z}_{k}^c \right) )  \nonumber \\
 &\leq \nabla  \psi(\mathbf{z}_{k+1}^c )^\top ( \mathbf{z}_{k+1}^c   -\Phi_k\left( \mathbf{z}_{k}^c  \right)) \nonumber \\
 & \qquad - \nabla \psi(\mathbf{\bar{z}}_{k+1})^\top (  \mathbf{z}_{k+1}^c  - \Phi_k\left( \mathbf{z}_{k}^c  \right) ) \nonumber\\
 &\leq M\Vert  \mathbf{z}_{k+1}^c - \Phi_k\left(\mathbf{z}_{k}^c  \right) \Vert .
 \end{align}
Substituting \eqref{toSubstitue1} and \eqref{toSubstitue2} into (\ref{instaFirst}) yields 
\begin{align}
   &\eta_k^0 \left( f_k (  \mathbf{\hat{z}}_k^0 ) -   f_k ( \mathbf{z}_{k}^c ) \right)
    \\&\leq D_\psi \left(  \mathbf{z}_{k}^c , \mathbf{\hat{z}}_k^0 \right) - D_\psi \left( \mathbf{z}_{k+1}^c ,   \mathbf{\hat{z}}_{k+1}^0\right)   +  \frac{(\eta_k^0)^2 }{2  \sigma  }  \Vert\nabla f_k ( \mathbf{\hat{z}}_k^0 )  \Vert^2 
  \nonumber  \\&  +    \sum_{i=1}^{\mathrm{it}(k)-1}   \frac{  (\eta_k^{i})^2 }{2 \sigma} \Vert \nabla f_k \left( \mathbf{\hat{z}}_k^{i} \right) \Vert^2        +  \sum_{i=1}^{\mathrm{it}(k)-1}  \eta_k^{i} \, \left( f_k(  \mathbf{z}_{k}^c )  -    f_k(\mathbf{\hat{z}}_k^{i}) \right)\nonumber \\
   & +  M  \Vert  \mathbf{z}_{k+1}^c - \Phi_k\left(\mathbf{z}_{k}^c  \right) \Vert. \nonumber
  \end{align} 
 Rearranging the above inequality yields
 \begin{align}
&  \sum_{i=0}^{\mathrm{it}(k)-1}\eta_k^{i} ( f_k( \mathbf{\hat{z}}_k^{i})    -  f_k( \mathbf{z}_{k}^c ) )    \\&\leq   D_\psi \left(  \mathbf{z}_{k}^c , \mathbf{\hat{z}}_k^0 \right) - D_\psi \left( \mathbf{z}_{k+1}^c ,   \mathbf{\hat{z}}_{k+1}^0\right)   \nonumber \\& \enskip +    \sum_{i=0}^{\mathrm{it}(k)-1}   \frac{  (\eta_k^{i})^2 }{2 \sigma} \Vert \nabla f_k \left( \mathbf{\hat{z}}_k^{i} \right) \Vert^2 + \!     M \Vert \mathbf{z}_{k+1}^c - \Phi_k\left(\mathbf{z}_{k}^c  \right) \Vert   \nonumber.
  \end{align}
  Since 
$  \min\limits_{0 \leq i \leq \mathrm{it}(k)} f_k (  \mathbf{\hat{z}}_k^{i} )  \,\sum_{i=0}^{\mathrm{it}(k)-1}  \eta_k^{i}    \leq \sum_{i=0}^{\mathrm{it}(k)-1}\eta_k^{i}\, f_k( \mathbf{\hat{z}}_k^{i})  
$, we obtain
 \begin{align}
& \left( \min_{0 \leq i \leq \mathrm{it}(k)} f_k (  \mathbf{\hat{z}}_k^{i} )  -  f_k (\mathbf{z}_{k}^c )\right)\sum_{i=0}^{\mathrm{it}(k)-1}  \eta_k^{i}   \\&\leq   D_\psi \left(  \mathbf{z}_{k}^c , \mathbf{\hat{z}}_k^0 \right) - D_\psi \left( \mathbf{z}_{k+1}^c ,   \mathbf{\hat{z}}_{k+1}^0\right)   \nonumber \\& \enskip  +    \sum_{i=0}^{\mathrm{it}(k)-1}   \frac{  (\eta_k^{i})^2 }{2 \sigma} G_f^2 +    M \Vert \mathbf{z}_{k+1}^c - \Phi_k\left(\mathbf{z}_{k}^c  \right) \Vert   \nonumber.
  \end{align}
  Dividing the latter inequality by $\sum_{i=0}^{\mathrm{it}(k)-1}  \eta_k^{i}  $ yields the desired result.
 \end{proof}  
  \subsection{Proof of Lemma \ref{regret_betweeniT_eta}}
   \label{proof: Lemma  regret_betweeniT_eta}
 \begin{proof}
  Since the gradients of $f_k$ are Lipschitz continuous by Assumption \ref{ass: f strngSmooth}, by (\ref{LipschitzForRegret}),
         \begin{align}
  f_k(\mathbf{z})& \geq     f_k \left(\mathbf{\hat{z}}_{k}^{i+1}   \right) + \frac{1}{2}  \big( \frac{\sigma}{\eta_k^i} - L_f \big) \left\Vert  \mathbf{\hat{z}}_{k}^{i+1}   - \mathbf{\hat{z}}_{k}^i \right\Vert^2    \\& \enskip +  \frac{1}{\eta_k^i} \big( D_\psi \big(\mathbf{z}, \mathbf{\hat{z}}_{k}^{i+1}  \big)- D_\psi \big(\mathbf{z},\mathbf{\hat{z}}_{k}^i\big)  \big),  \nonumber 
   \end{align}
   for all $\mathbf{z} \in \mathcal{S}_k$.    Since the step size satisfies $ \frac{\sigma}{\eta_k^i} - L_f \geq   0   $, we obtain 
       \begin{align}
  f_k(\mathbf{z})& \geq     f_k \left(\mathbf{\hat{z}}_{k}^{i+1}   \right)  +  \frac{1}{\eta_k^i} \big( D_\psi \big(\mathbf{z}, \mathbf{\hat{z}}_{k}^{i+1}  \big)- D_\psi \big(\mathbf{z},\mathbf{\hat{z}}_{k}^i\big)  \big) .   
   \end{align}
      Thus,  for $\mathbf{z}  = \mathbf{z}_{k}^c\in \mathcal{S}_k $, we have that 
         \begin{align}
  \eta_k^i \left( f_k \left(\mathbf{\hat{z}}_{k}^{i+1}   \right) - f_k(\mathbf{z}_{k}^c) \right) \leq     D_\psi \big(\mathbf{z}_{k}^c,\mathbf{\hat{z}}_{k}^i\big)    -       D_\psi \big(\mathbf{z}_{k}^c, \mathbf{\hat{z}}_{k}^{i+1} \big)   .   
   \end{align}
To prove the second statement in Lemma \ref{regret_betweeniT_eta}, suppose we execute $\mathrm{it}(k)$ iterations at the time instant $k$.  Evaluating the latter inequality for $i=0$ yields  
  \begin{subequations}
 \begin{align}
  \eta_k^0 &\left(f_k(\mathbf{\hat{z}}_{k}^1) - f_k( \mathbf{z}_{k}^c) \right)  \\ &\leq  D_\psi( \mathbf{z}_{k}^c,\mathbf{\hat{z}}_{k}^{0}) - D_\psi(\mathbf{z}_{k}^c, \mathbf{\hat{z}}_{k}^{1})  \nonumber 
    \\&=   D_\psi \left(  \mathbf{z}_{k}^c , \mathbf{\hat{z}}_k^0 \right) - D_\psi \left( \mathbf{z}_{k+1}^c ,   \mathbf{\hat{z}}_{k+1}^0\right)   + R_1 + R_2   \nonumber
  \end{align}
  where 
  \begin{align}
      R_1 &\coloneqq    D_\psi  \left(  \Phi_k\left(  \mathbf{z}_{k}^c \right)  ,  \mathbf{\bar{z}}_{k+1}  \right)   - D_\psi  \left( \mathbf{z}_{k}^c ,    \mathbf{\hat{z}}_k^1 \right)    \label{newrh3ad} \\
      R_2 &\coloneqq  D_\psi \left( \mathbf{z}_{k+1}^c ,     \mathbf{\hat{z}}_{k+1}^0 \right)  -D_\psi  \left(   \Phi_k\left( \mathbf{z}_{k}^c \right)   ,     \mathbf{\bar{z}}_{k+1}  \right)    \label{newrh4ad} .
  \end{align}
  \end{subequations}
Again,  we can compute an upper bound for each of these terms as follows.  
Since $\mathbf{\bar{z}}_{k+1} =\Phi_k\big(  \mathbf{\hat{z}}_k^{\mathrm{it}(k)} \big)$, by (\ref{Dcontadap}),
 \begin{align}
 R_1& =        D_\psi  \big( \Phi_k\left(  \mathbf{z}_{k}^c \right)   ,  \Phi_k\big(  \mathbf{\hat{z}}_k^{\mathrm{it}(k)} \big)\big)   - D_\psi  \left( \mathbf{z}_{k}^c ,    \mathbf{\hat{z}}_k^1 \right)  \\
 & \leq    D_\psi  \big(      \mathbf{z}_{k}^c    ,    \mathbf{\hat{z}}_k^{\mathrm{it}(k)} \big)     -D_\psi  \left( \mathbf{z}_{k}^c ,    \mathbf{\hat{z}}_k^1 \right) \nonumber.
 \end{align}
 Moreover, employing (\ref{result: regret_betweeniT_eta}) in Lemma \ref{regret_betweeniT_eta} for each iteration step  starting from $i=\mathrm{it}(k)-1$ to $i=2$ yields
\begin{align}
R_1&  \leq    
      D_\psi  \big(   \mathbf{z}_{k}^c    ,    \mathbf{\hat{z}}_k^{\mathrm{it}(k) -1}  \big)      \\
& \enskip +   \eta_k^{\mathrm{it}(k)-1} \big( f_k( \mathbf{z}_{k}^c) -f_k\big(\mathbf{\hat{z}}_k^{\mathrm{it}(k) } \big)  \big)   -D_\psi  \left( \mathbf{z}_{k}^c ,    \mathbf{\hat{z}}_k^1 \right) \nonumber\\
  &\leq   D_\psi \big(  \mathbf{z}_{k}^c     ,    \mathbf{\hat{z}}_k^{\mathrm{it}(k)-2}  \big)      +   \eta_k^{\mathrm{it}(k)-2} \big( f_k( \mathbf{z}_{k}^c ) -f_k\big( \mathbf{\hat{z}}_k^{\mathrm{it}(k)-1} \big) \big) \nonumber \\& \enskip +  \eta_k^{\mathrm{it}(k)-1} \big( f_k(  \mathbf{z}_{k}^c ) -f_k\big( \mathbf{\hat{z}}_k^{\mathrm{it}(k)} \big)  \big)   -D_\psi  \left( \mathbf{z}_{k}^c ,    \mathbf{\hat{z}}_k^1 \right) \nonumber \\
  & \leq \dots \nonumber \\
    &\leq       \sum_{i=2}^{\mathrm{it}(k) }  \eta_k^{i-1} \, \left( f_k(  \mathbf{z}_{k}^c )  -    f_k(\mathbf{\hat{z}}_k^{i}) \right) \nonumber  .  
 \end{align}
 Note that $R_2=T_2$ in (\ref{newrh4}). Hence, by (\ref{toSubstitue2}), we obtain  
\begin{align}
     \eta_k^0 \left(f_k(\mathbf{\hat{z}}_{k}^1) - f_k( \mathbf{z}_{k}^c) \right) &\leq D_\psi \left(  \mathbf{z}_{k}^c , \mathbf{\hat{z}}_k^0 \right) - D_\psi \left( \mathbf{z}_{k+1}^c ,   \mathbf{\hat{z}}_{k+1}^0\right)     \nonumber  \\&\quad      +  \sum_{i=2}^{\mathrm{it}(k) }  \eta_k^{i-1} \, \left( f_k(  \mathbf{z}_{k}^c )  -    f_k(\mathbf{\hat{z}}_k^{i}) \right)  \nonumber\\
   &\quad +  M \Vert  \mathbf{z}_{k+1}^c - \Phi_k\left(\mathbf{z}_{k}^c  \right) \Vert.
  \end{align} 
 Rearranging the above inequality yields
 \begin{align}
&  \sum_{i=1}^{\mathrm{it}(k)}\eta_k^{i-1} ( f_k( \mathbf{\hat{z}}_k^{i})    -  f_k( \mathbf{z}_{k}^c ) )      \\&= \eta_k^0 \,   f_k ( \mathbf{\hat{z}}_k^1 )    - \eta_k^0 \,  f_k (\mathbf{z}_{k}^c ) + \sum_{i=2}^{\mathrm{it}(k) }  \eta_k^{i-1} \, (f_k( \mathbf{\hat{z}}_k^{i})- f_k( \mathbf{z}_{k}^c) )
   \nonumber \\&\leq   D_\psi \left(  \mathbf{z}_{k}^c , \mathbf{\hat{z}}_k^0 \right) - D_\psi \left( \mathbf{z}_{k+1}^c ,   \mathbf{\hat{z}}_{k+1}^0\right)  + \!   M \Vert \mathbf{z}_{k+1}^c - \Phi_k\left(\mathbf{z}_{k}^c  \right) \Vert   \nonumber.
  \end{align}
  Since  
   \begin{align}
 \min_{0 \leq i \leq \mathrm{it}(k) } f_k (  \mathbf{\hat{z}}_k^{i} ) \enskip   \sum_{i=0}^{\mathrm{it}(k)-1}  \eta_k^{i}  
& \leq  \min_{1 \leq i \leq \mathrm{it}(k) } f_k (  \mathbf{\hat{z}}_k^{i} ) \enskip   \sum_{i=1}^{\mathrm{it}(k)}  \eta_k^{i-1}  
   \nonumber \\&   \leq \sum_{i=1}^{\mathrm{it}(k)}\eta_k^{i-1} \enskip f_k( \mathbf{\hat{z}}_k^{i}),    
  \end{align}
we obtain
 \begin{align}
& \left( \min_{0 \leq i \leq \mathrm{it}(k)} f_k (  \mathbf{\hat{z}}_k^{i} )  -  f_k (\mathbf{z}_{k}^c )\right)\sum_{i=0}^{\mathrm{it}(k)-1}  \eta_k^{i}   \\& \enskip   \leq   D_\psi \left(  \mathbf{z}_{k}^c , \mathbf{\hat{z}}_k^0 \right) - D_\psi \left( \mathbf{z}_{k+1}^c ,   \mathbf{\hat{z}}_{k+1}^0\right)    +     M  \Vert \mathbf{z}_{k+1}^c - \Phi_k\left(\mathbf{z}_{k}^c  \right) \Vert   \nonumber.
  \end{align}
  Dividing the latter inequality by $\sum_{i=0}^{\mathrm{it}(k)-1}  \eta_k^{i}  $ yields the desired result.
 \end{proof}